\documentclass{article}
\usepackage{amsmath}
\usepackage{amssymb}
\usepackage{amsthm}
\usepackage{enumerate}
\usepackage{wasysym}
\usepackage{dsfont}
\usepackage{comment}
\usepackage{tipa}
\usepackage{tikz}
\usepackage{tikz-cd}

\newtheorem{definition}{Definition}[section]
\newtheorem{theorem}[definition]{Theorem}
\newtheorem{lemma}[definition]{Lemma}

\newtheorem{sub-claim}[definition]{Sub-Claim}
\newtheorem{corollary}[definition]{Corollary}
\newtheorem*{claim*}{Claim}

\newcommand{\T}{\mathbb{T}}
\newcommand{\la}{\langle}
\newcommand{\ra}{\rangle}

\newcommand{\Q}{\mathsf{Q}}
\newcommand{\Term}{\mathsf{Term}}
\newcommand{\Rack}{\mathsf{Rack}}
\newcommand{\Quandle}{\mathsf{Quandle}}
\newcommand{\xbold}{\mathsf{x}}

\newcommand{\Conj}{\mathsf{Conj}}
\newcommand{\E}{\mathsf{E}}
\newcommand{\y}{\mathsf{y}}
\newcommand{\Grp}{\mathsf{Grp}}
\newcommand{\HalfConj}{\mathsf{HalfConj}}
\newcommand{\Left}{\mathsf{Left}}
\newcommand{\z}{\mathsf{z}}
\newcommand{\W}{\mathsf{W}}
\newcommand{\R}{\mathsf{R}}
\newcommand{\F}{\mathsf{F}}
\newcommand{\Z}{\mathbb{Z}}
\newcommand{\Tmod}{\T\mathsf{mod}}
\newcommand{\cod}{\mathsf{cod}}
\newcommand{\dom}{\mathsf{dom}}
\newcommand{\id}{\mathsf{id}}
\newcommand{\Inn}{\mathsf{Inn}}
\newcommand{\Aut}{\mathsf{Aut}}
\newcommand{\Id}{\mathsf{Id}}
\newcommand{\End}{\mathsf{End}}
\newcommand{\Group}{\mathsf{Group}}

\begin{document}

\title{Isotropy Groups of Free Racks and Quandles}
\author{Jason Parker}

\maketitle

\begin{abstract}
In this article, we apply the methods developed in \cite{MFPSpaper}, \cite{thesis} to characterize the (covariant) \emph{isotropy groups} of free, finitely generated racks and quandles. As a consequence, we show that the usual inner automorphisms of such racks and quandles are precisely those automorphisms that are `coherently extendible'. We then use this result to compute the \emph{global isotropy groups} of the categories of racks and quandles, i.e. the automorphism groups of the identity functors of these categories.   
\end{abstract}

\section{Introduction}

In \cite[Theorem 1]{Bergman}, George Bergman proved that the usual inner automorphisms of a group $G$ (defined in terms of conjugation) are \emph{exactly} those automorphisms of $G$ that can be \emph{coherently extended} along morphisms out of $G$. More precisely, he showed that an automorphism $\alpha : G \xrightarrow{\sim} G$ is inner \emph{if and only if} for any group homomorphism $f : G \to H$ with domain $G$, one can define a group automorphism $\pi_f : H \xrightarrow{\sim} H$ of the codomain in such a way that the resulting family of automorphisms $(\pi_f)_f$ is \emph{natural}, meaning that if $f : G \to G'$ and $f' : G' \to G''$ are group homomorphisms, then the following square commutes: 
\begin{center}
\begin{tikzcd}[ampersand replacement=\&, row sep = huge, column sep = huge]
G' \arrow{r}{\pi_{f}}\arrow{d}[swap]{f'} \& G'\arrow{d}{f'} \\
G''\arrow{r}[swap]{\pi_{f' \circ f}} \& G'' \\
\end{tikzcd}
\end{center} 
\noindent Such a family of automorphisms may be equivalently described as a natural automorphism of the projection functor $G/\Group \to \Group$, from the slice category under the group $G$ to the category $\Group$. We refer to such a family of automorphisms as an \emph{extended inner automorphism} of $G$. If $\mathcal{Z}(G)$ denotes the group of all extended inner automorphisms of $G$, i.e. the group of all natural automorphisms of the projection functor $G/\Group \to \Group$, then Bergman also showed in \cite[Theorem 2]{Bergman} that $\mathcal{Z}(G)$ is isomorphic to the group $G$ itself.   

In \cite{MFPSpaper}, the author and his collaborators took inspiration from this result of Bergman to analyze the extended inner automorphisms of the models of \emph{any} (single-sorted) algebraic or equational theory whatsoever (as a special case, a group is a model of the algebraic theory of groups). In \cite{thesis}, the author then extended this study from (single-sorted) algebraic theories to (multi-sorted) \emph{essentially} algebraic theories (where operations are only required to be \emph{partially} defined). 

As for groups, there is a well-known notion of \emph{inner automorphism} for a rack or quandle (which, along with the theories of racks and quandles, is defined and discussed in the subsequent section). Using techniques developed in \cite{MFPSpaper} and \cite{thesis}, we will prove a result for \emph{free} racks and quandles that is analogous to the one that George Bergman proved for (arbitrary) groups: namely, we will show that an automorphism of a rack/quandle is inner (in the well-known sense) \emph{if and only if} it can be `coherently extended' along morphisms out of the rack/quandle, in the sense described for groups.

\section{Background}    

We now review the relevant material from \cite{MFPSpaper} and \cite[Chapter 2]{thesis} on the isotropy groups of free models of algebraic theories. For more details, see those references.

 A single-sorted \emph{algebraic theory} $\T$ is a set of equations between the terms of a single-sorted first-order signature $\Sigma$ consisting of operation symbols. For example, the theories of (commutative) monoids, (abelian) groups, (commutative) rings with unit, and the theories of racks and quandles (to be defined explicitly below). A (set-based) \emph{model} $M$ of a single-sorted algebraic theory $\T$ is a set equipped with functions on $M$ interpreting the function symbols of the signature, which satisfies the axioms of $\T$. One can then form the category $\Tmod$ of (set-based) models of $\T$ and homomorphisms between them. 
 
Given $M \in \Tmod$, the \emph{(covariant) isotropy group} of $M$ is the group $\mathcal{Z}_\T(M)$ of all natural automorphisms of the forgetful functor $M / \Tmod \to \Tmod$. More concretely, an element of $\mathcal{Z}_\T(M)$ is a family of automorphisms
\[ \pi = \left(\pi_h : \cod(h) \xrightarrow{\sim} \cod(h)\right)_{\dom(h) = M} \] in $\Tmod$ indexed by morphisms $h \in \Tmod$ with domain $M$ that has the following naturality property: if $h : M \to M'$ and $h' : M' \to M''$ are homomorphisms in $\Tmod$, then the following diagram commutes:
\begin{center}
\begin{tikzcd}[ampersand replacement=\&, row sep = huge, column sep = huge]
M' \arrow{r}{\pi_{h}}\arrow{d}[swap]{h'} \& M'\arrow{d}{h'} \\
M''\arrow{r}[swap]{\pi_{h' \circ h}} \& M'' \\
\end{tikzcd}
\end{center}     
We then say that an automorphism $h : M \xrightarrow{\sim} M$ of $M \in \Tmod$ is a \emph{categorical inner automorphism} (or is \emph{coherently extendible}) if there is some $\pi \in \mathcal{Z}_\T(M)$ with $h = \pi_{\mathsf{id}_M} : M \xrightarrow{\sim} M$; roughly, $h$ is a categorical inner automorphism if it can be coherently extended along any morphism out of $M$. 

In \cite{MFPSpaper} and \cite{thesis} the author and his collaborators gave a \emph{logical} characterization of the isotropy group of a model of an algebraic theory. We will only review the (simpler) characterization for the \emph{free, finitely generated} models, as these are the only models that will concern us in this article. If $\T$ is an algebraic theory and $n \geq 0$, then a model $M_n \in \Tmod$ is \emph{free on} $n$ \emph{generators} if it contains $n$ distinct elements (the \emph{generators}) $m_1, \ldots, m_n$ and has the following universal property: for any $N \in \Tmod$ and any elements $a_1, \ldots, a_n \in N$, there is a unique homomorphism $h_{a_1, \ldots, a_n} : M_n \to N$ with $h_{a_1, \ldots, a_n}(m_i) = a_i$ for each $1 \leq i \leq n$. For a fixed $n \geq 0$, the free $\T$-models on $n$ generators are all isomorphic, so we may speak of \emph{the} free $\T$-model on $n$ generators (unique up to isomorphism). 

The free $\T$-model on $n$ generators can be given the following explicit description: let $\{\y_1, \ldots, \y_n\}$ be a set of $n$ distinct constants (not in $\Sigma$), and let $\Sigma(\y_1, \ldots, \y_n)$ be the single-sorted signature obtained from $\Sigma$ by adding the elements $\y_1, \ldots, \y_n$ as new constant symbols. Let $\T(\y_1, \ldots, \y_n)$ be the algebraic theory with the same axioms as $\T$, but now regarded as an algebraic theory over the signature $\Sigma(\y_1, \ldots, \y_n)$. Consider the set $\Term^c(\Sigma(\y_1, \ldots, \y_n))$ of \emph{closed} terms over the signature $\Sigma(\y_1, \ldots, \y_n)$. We then define a relation $\sim_{\T, n} \ = \ \sim_\T$ on $\Term^c(\Sigma(\y_1, \ldots, \y_n))$ by setting $s \sim_{\T} t$ iff \[ \T(\y_1, \ldots, \y_n) \vdash s = t \] for any $s, t \in \Term^c(\Sigma(\y_1, \ldots, \y_n))$. Roughly, we have $s \sim_{\T} t$ iff $s$ can be proved equal to $t$ using (only) the axioms of $\T$. Then $\sim_{\T}$ is a \emph{congruence} relation on $\Term^c(\Sigma(\y_1, \ldots, \y_n))$, i.e. an equivalence relation that is compatible with the operation symbols in $\Sigma$. We can then form the \emph{quotient} $\T$-model \[ \Term^c(\Sigma(\y_1, \ldots, \y_n))/{\sim_{\T}}, \] whose objects are $\sim_{\T}$-congruence classes, which will have the desired universal property, with generators $[\y_1], \ldots, [\y_n]$. So we can take \[ M_n :=\Term^c(\Sigma(\y_1, \ldots, \y_n))/{\sim_{\T}} \] as an explicit construction of the free $\T$-model on $n$ generators.

Now let $G_\T(M_n)$ be the set of all elements
\[ [t] \in \Term^c(\Sigma(\xbold, \y_1, \ldots, \y_n))/{\sim_{\T}} \] (note the additional constant $\xbold$) with the following properties:
\begin{itemize}
\item $[t]$ is \emph{invertible}, meaning that there is some $t^{-1} \in \Term^c(\Sigma(\xbold, \y_1, \ldots, \y_n))$ such that
\[ \T(\xbold, \y_1, \ldots, \y_n) \vdash t[t^{-1}/\xbold] = \xbold = t^{-1}[t/\xbold]. \]
\item $[t]$ \emph{commutes generically with} the operation symbols of $\Sigma$, meaning that if $f$ is an $m$-ary operation symbol of $\Sigma$, then
\[ \T(\xbold_1, \ldots, \xbold_m, \y_1, \ldots, \y_n) \vdash t[f(\xbold_1, \ldots, \xbold_m)/\xbold] = f(t[\xbold_1/\xbold], \ldots, t[\xbold_m/\xbold]). \] 
\end{itemize}  
Then this set $G_\T(M_n)$ can be given the structure of a group (with unit element $[\xbold]$ and multiplication given by substitution into $\xbold$), and we then have (cf. \cite[Corollary 2.4.15]{thesis})
\[ \mathcal{Z}_\T(M_n) \cong G_\T(M_n). \] We refer to $G_\T(M_n)$ as the \emph{logical} isotropy group of $M_n$; thus, the (categorical) covariant isotropy group of $M_n$ is isomorphic to its logical isotropy group. Therefore, the extended inner automorphisms of $M$ can essentially be identified with those (congruence classes of) closed $\Sigma$-terms over the constant $\xbold$ and the generating constants $\y_1, \ldots, \y_n$ that are invertible and commute generically with the operations of $\Sigma$. 

Given any $[t] \in \Term^c(\Sigma(\xbold, \y_1, \ldots, \y_n))/{\sim_{\T}}$ and any $\T$-model $N$ with $n$ distinct elements $a_1, \ldots, a_n \in N$, the element $[t]$ induces a function
\[ [t]^{N, a_1, \ldots, a_n} : N \to N; \] roughly, given any $a \in N$, the value $[t]^{N, a_1, \ldots, a_n}(a) \in N$ is the element of $N$ obtained by substituting $a_1, \ldots, a_n$ for $\y_1, \ldots, \y_n$ and $a$ for $\xbold$ in $t$, and then interpreting/evaluating the result in $N$. The following results then follow from the definition of the isomorphism $\mathcal{Z}_\T(M_n) \cong G_\T(M_n)$, cf. \cite[Corollary 2.2.42]{thesis}. For any homomorphism $h : M_n \to N$ in $\Tmod$, let us write $h_1, \ldots, h_n \in N$ for the images $h([\y_1]), \ldots, h([\y_n]) \in N$ of the generators of $M_n$ under $h$.
\begin{itemize}
\item Given any (not necessarily \emph{natural}) family 
\[ \pi = \left(\pi_h : \cod(h) \to \cod(h)\right)_{\dom(h) = M_n} \] of \emph{endo}morphisms in $\Tmod$ indexed by morphisms with domain $M_n$, we have $\pi \in \mathcal{Z}_\T(M_n)$ iff there is some (uniquely determined) element $[t] \in G_\T(M_n)$ with the property that
\[ \pi_h = [t]^{N, h_1, \ldots, h_n} : N \to N \] for each homomorphism $h : M_n \to N$ in $\Tmod$ with domain $M_n$ (in particular, every such function $[t]^{N, h_1, \ldots, h_n}$ will be a $\T$-model automorphism).    
\item Given any endomorphism $h : M_n \to M_n$ in $\Tmod$, we have that $h$ is a \emph{categorical inner automorphism} of $M_n$ iff there is some element $[t] \in G_\T(M_n)$ with \[ h = [t]^{M, \id_1, \ldots, \id_n} : M \to M \] (where $\id : M_n \to M_n$ is the identity morphism). 
\end{itemize}            

In this article, we will be computing the (logical) isotropy groups of the free, finitely generated models of the algebraic theories of \emph{racks} and \emph{quandles}. Our computations will depend heavily on the solutions of the word problems for free racks and quandles in terms of the solution of the word problem for free groups given in \cite[Section 4.1]{SDworld}. We now review the definitions of these algebraic theories: 
    
\begin{definition} 
{\em
\
\begin{enumerate}
\item Let $\Sigma$ be the single-sorted signature containing two binary function symbols $\lhd, \lhd^{-1}$, written in infix notation.
\item Let $\T_\Rack$ be the algebraic theory over the signature $\Sigma$ with the following axioms (where $x, y, z$ are variables):
\begin{itemize}
\item $x \lhd (y \lhd z) = (x \lhd y) \lhd (x \lhd z)$.
\item $x \lhd^{-1} (y \lhd^{-1} z) = (x \lhd^{-1} y) \lhd^{-1} (x \lhd^{-1} z)$.
\item $(x \lhd y) \lhd^{-1} y = x$.
\item $(x \lhd^{-1} y) \lhd y = x$.
\end{itemize}
\item Let $\T_\Quandle$ be the algebraic theory over the signature $\Sigma$ whose axioms are those of $\T_\Rack$ together with the following two additional axioms:
\begin{itemize}
\item $x \lhd x = x$. 
\item $x \lhd^{-1} x = x$.
\end{itemize}
\end{enumerate} \qed
} 
\end{definition}

Racks and quandles are algebraic structures that originally arose in the context of knot theory, to describe so-called \emph{Reidemeister moves}.  Algebraically speaking, they axiomatize the notion of (group) \emph{conjugation} (without reference to multiplication or inverses). For example, a canonical quandle structure can be defined on (the underlying set of) any group $G$ by setting $g \lhd h := ghg^{-1}$ and $g \lhd^{-1} h := g^{-1}hg$. Examples of racks that are \emph{not} (necessarily) quandles include constant actions $x \lhd y := \sigma(x) := x \lhd^{-1} y$, where $\sigma$ is a permutation of a fixed set $X$. Any equation involving only conjugation that is provable in the theory of groups is also provable in the theory of quandles, so that quandles essentially axiomatize the concept of group-theoretic conjugation (\cite[Theorem 4.2]{Joyce}).

If $R$ is any rack with $r \in R$, then it can be shown that the function $(-) \lhd r : R \to R$ is a rack automorphism, with inverse $(-) \lhd^{-1} r : R \to R$. Let $\Aut(R)$ be the group of all rack automorphisms of $R$. Then the group $\mathsf{Inn}(R)$ of \emph{algebraic inner automorphisms} of $R$ is defined to be the subgroup of $\Aut(R)$ generated by all such rack automorphisms $(-) \lhd r$. If $Q$ is a quandle, then we define $\Inn(Q)$ analogously. Explicitly, an automorphism $f : R \xrightarrow{\sim} R$ of a rack $R$ is an \emph{algebraic inner} automorphism iff there are $p, n \geq 0$ and $r_1, \ldots, r_n \in R$ and $\delta_1, \ldots, \delta_p, \epsilon_1, \ldots, \epsilon_n = \pm 1$ such that \[ f(r) = (\ldots((((\ldots((r \lhd^{\delta_1} r) \lhd^{\delta_2} r) \ldots) \lhd^{\delta_p} r) \lhd^{\epsilon_1} r_1) \lhd^{\epsilon_2} r_2) \ldots) \lhd^{\epsilon_n} r_n \] for all $r \in R$. Similarly, an automorphism $f : Q \xrightarrow{\sim} Q$ of a quandle $Q$ is an \emph{algebraic inner} automorphism iff there are $n \geq 0$ and $q_1, \ldots, q_n \in Q$ and $\epsilon_1, \ldots, \epsilon_n = \pm 1$ such that \[ f(q) = (\ldots((q \lhd^{\epsilon_1} q_1) \lhd^{\epsilon_2} q_2) \ldots) \lhd^{\epsilon_n} q_n \] for all $q \in Q$. It is not difficult to show that any algebraic inner automorphism of an arbitrary rack or quandle is also a \emph{categorical} inner automorphism. As a consequence of our results in the next two sections, we will show that the converse is true for \emph{free, finitely generated} racks and quandles: i.e. we will show that any \emph{categorical} inner automorphism of a free, finitely generated rack or quandle must be an \emph{algebraic} inner automorphism as well.      

\section{Isotropy Groups of Free Quandles}

We will first characterize the (logical) isotropy groups of free, finitely generated quandles, because this turns out to be the simpler task. First, we give an explicit description of the free quandle on $n$ generators, as given in e.g. \cite[Proposition 4.2]{SDworld}.

Let $\Sigma_\Grp$ be the single-sorted signature of the algebraic theory $\T_\Grp$ of groups, with  three function symbols $\cdot$ (binary), $^{-1}$ (unary), and $e$ (constant), the first two written in infix notation. For any (finite) set $X$, let $\Sigma_\Grp(X)$ be the signature that extends the signature $\Sigma_\Grp$ by adding the elements of $X$ as new constants, and let $\Term^c(\Sigma_\Grp(X))$ be the set of \emph{closed} terms over the signature $\Sigma_\Grp(X)$. Also let $\T_\Grp(X)$ be the theory with the same axioms as $\T_\Grp$, but now regarded as an algebraic theory over the signature $\Sigma_\Grp(X)$. 

Given a (finite) set $X$, it has been shown (cf. \cite[Proposition 4.2]{SDworld}) that the free quandle on $X$ has the following presentation, which we denote by $\Conj(\F_X)$, where $\F_X$ is the free group on $X$ (with the presentation given in Section 1 for $\T = \T_\Grp$). The underlying set of $\Conj(\F_X)$ is just the underlying set of $\F_X$, and for any $[s], [t] \in \F_X$ (so that $s, t \in \Term^c(\Sigma_\Grp(X))$) we have \[ [s] \lhd [t] := [t^{-1} \cdot s \cdot t] \in \F_X \] and \[ [s] \lhd^{-1} [t] := \left[t \cdot s \cdot t^{-1}\right] \in \F_X. \] (When brackets are omitted when writing group multiplication, we will assume that they associate to the left).

Let $\Sigma(X)$ be the signature that extends $\Sigma$ (the signature for racks and quandles) by adding the elements of $X$ as constants, and let $\Term^c(\Sigma(X))$ be the set of \emph{closed} terms over the signature $\Sigma(X)$. Then there is a function \[ \E_X : \Term^c(\Sigma(X)) \to \Term^c(\Sigma_\Grp(X)) \] defined by induction on the structure of closed terms by
\[ \E_X(\xbold) := \xbold \tag{$\xbold \in X$} \]
\[ \E_X(s \lhd t) := \E(t)^{-1}\cdot \E(s) \cdot \E(t) \]
\[ \E_X(s \lhd^{-1} t) := \E(t) \cdot \E(s) \cdot \E(t)^{-1} \] for $s, t \in \Term^c(\Sigma(X))$. 
Technically the map $\E_X$ depends on $X$, but we will omit the subscript when confusion will not arise. We then have the following substitution lemma:

\begin{lemma}
Let $X$ be an arbitrary (finite) set with designated element $\xbold \in X$, and let $X'$ be another (finite) set with $X \subseteq X'$. Then for any $t \in \Term^c(\Sigma(X))$ and $s \in \Term^c(\Sigma(X'))$, we have \[ \E(t[s/\xbold]) \sim_{\Grp} \E(t)[\E(s)/\xbold] \] in $\F_{X'}$. 
\end{lemma}

\begin{proof}
For a fixed $s \in \Term^c(\Sigma(X'))$, we prove the claim by induction on $t \in \Term^c(\Sigma(X))$: 
\begin{itemize}
\item If $t \equiv \xbold$, then we have
\begin{align*}
\E(t[s/\xbold]) &= \E(\xbold[s/\xbold]) \\
		 &= \E(s) \\
		 &= \xbold[\E(s)/\xbold] \\
		 &= \E(\xbold)[\E(s)/\xbold] \\
		 &= \E(t)[\E(s)/\xbold].
\end{align*}

\item If $t \equiv \y$ for some $\y \in X$ with $\xbold \neq \y$, then we have
\begin{align*}
\E(t[s/\xbold]) &= \E(\y[s/\xbold]) \\
		 &= \E(\y) \\
		 &= \y \\
		 &= \y[\E(s)/\xbold] \\
		 &= \E(\y)[\E(s)/\xbold] \\
		 &= \E(t)[\E(s)/\xbold].
\end{align*}

\item Suppose that $t \equiv t_1 \lhd t_2$ for some $t_1, t_2 \in \Term^c(\Sigma(X))$ with 
\[ \E(t_1[s/\xbold]) \sim \E(t_1)[\E(s)/\xbold] \]
and
\[ \E(t_2[s/\xbold]) \sim \E(t_2)[\E(s)/\xbold]. \]
Then we have
\begin{align*}
\E(t[s/\xbold])	 &= \E((t_1 \lhd t_2)[s/\xbold]) \\
		 &= \E(t_1[s/\xbold] \lhd t_2[s/\xbold]) \\
		 &= \E(t_2[s/\xbold])^{-1} \cdot \E(t_1[s/\xbold]) \cdot \E(t_2[s/\xbold]) \\
		 &\sim (\E(t_2)[\E(s)/\xbold])^{-1} \cdot \E(t_1)[\E(s)/\xbold] \cdot \E(t_2)[\E(s)/\xbold] \\
		 &\sim \E(t_2)^{-1}[\E(s)/\xbold] \cdot \E(t_1)[\E(s)/\xbold] \cdot \E(t_2)[\E(s)/\xbold] \\
		 &= \left(\E(t_2)^{-1} \cdot \E(t_1) \cdot \E(t_2)\right)[\E(s)/\xbold] \\
		 &= \E(t_1 \lhd t_2)[\E(s)/\xbold] \\
		 &= \E(t)[\E(s)/\xbold];
\end{align*}
note that the fifth equality holds because for any $u, v \in \Term^c(\Sigma_\Grp(X'))$ we have $u[v/\xbold]^{-1} \sim u^{-1}[v/\xbold]$, as one can easily prove by induction on $u$ for a fixed $v$.  
The case where $t \equiv t_1 \lhd^{-1} t_2$ is exactly analogous.
\end{itemize}
\end{proof}

\noindent In \cite[Section 4.1]{SDworld}, the following result was proven:

\begin{theorem}[Dehornoy \cite{SDworld}]
For any (finite) set $X$ and $s, t \in \Term^c(\Sigma(X))$, we have \[ \T_{\Quandle}(X) \vdash s = t \] iff \[ \E(s) \sim \E(t). \] \qed
\end{theorem}

\noindent We will also require the following lemma. Recall that a closed term $s \in \Term^c(\Sigma_\Grp(X))$ is said to be \emph{reduced} if it is either $e$ or $\xbold \in X$, or else has the form $t_1^{\epsilon_1} \cdot \ldots \cdot t_m^{\epsilon_m}$ with $m \geq 2$ and $t_1, \ldots, t_m \in X$ and $\epsilon_1, \ldots \epsilon_m \in \{1, -1\}$ and $t_i = t_{i+1}$ implies $\epsilon_i = \epsilon_{i+1}$ for each $1 \leq i < m$. We will sometimes refer to such reduced terms as \emph{reduced group words} (over $X$). It is a standard fact about free groups that if $s, t$ are reduced group words over $X$ with $s \sim_\Grp t$, then $s \equiv t$ (i.e. $s$ and $t$ must be the same word).  

\begin{lemma}
Let $X := \{\xbold_0, \xbold_1, \y_1, \ldots, \y_n\}$ for some $n \geq 0$. If $s \in \Term^c(\Sigma_\Grp(\xbold, \y_1, \ldots, \y_n))$ is reduced and \[ s[\xbold_1/\xbold] \cdot s[\xbold_1^{-1}\xbold_0\xbold_1/\xbold] \sim s[\xbold_0/\xbold] \cdot s[\xbold_1/\xbold], \] then $s$ has at most one occurrence of $\xbold$ (which must then have exponent $1$). 
\end{lemma}

\begin{proof}
Assume the hypothesis. Then it is not difficult to see that if $s$ had at least two occurrences of $\xbold$, then the reduction of the term on the left side would have an occurrence of $\xbold_1$ to the left of an occurrence of $\xbold_0$, while the reduction of the term on the right side would have all occurrences of $\xbold_1$ to the right of all occurrences of $\xbold_0$. But if these terms are congruent (modulo $\sim_\Grp$), then their reductions must be equal, which is impossible, as just shown. So $s$ has at most one occurrence of $\xbold$.  

Lastly, \emph{if} the reduced term $s$ contains an occurrence of $\xbold$, then this unique occurrence must have exponent $1$. For suppose otherwise; then $s \equiv t_1 \cdot \xbold^{-1} \cdot t_2$ for some reduced (possibly empty) words $t_1, t_2 \in \Term^c(\Sigma_\Grp(\y_1, \ldots, \y_n))$. By the assumed congruence, we then have 
\[ t_1 \cdot \xbold_1^{-1} \cdot t_2 \cdot t_1 \cdot \xbold_1^{-1} \cdot \xbold_0^{-1} \cdot \left(\xbold_1^{-1}\right)^{-1} \cdot t_2 \sim t_1 \cdot \xbold_0^{-1} \cdot t_2 \cdot t_1 \cdot \xbold_1^{-1} \cdot t_2. \] However, even if $t_2 \cdot t_1 \sim e$, it is easy to see that the reductions of these terms will not be identical, and hence these terms cannot be congruent (modulo $\sim$). Therefore, if the reduced term $s$ contains an occurrence of $\xbold$, then this occurrence must have exponent $1$.
\end{proof}

\noindent We can now characterize the isotropy group of the free quandle on $n$ generators $\y_1, \ldots, \y_n$. Recall from Section 1 that if $\mathsf{Q}_n$ is the free quandle on $n$ generators $\y_1, \ldots, \y_n$, then the logical isotropy group $G_{\T_\Quandle}(\mathsf{Q}_n)$ of $\mathsf{Q}_n$ is (isomorphic to) the group of all elements
\[ [t] \in \Term^c(\Sigma({\xbold, \y_1, \ldots, \y_n}))/{\sim_\Quandle} \] that are invertible and commute generically with the function symbols $\lhd, \lhd^{-1}$, in the sense that there is some $s \in \Term^c(\Sigma({\xbold, \y_1, \ldots, \y_n}))$ with
\[ \T_\Quandle(\xbold, \y_1, \ldots, \y_n) \vdash t[s/\xbold] = \xbold = s[t/\xbold], \]
\[ \T_\Quandle(\xbold_0, \xbold_1, \y_1, \ldots, \y_n) \vdash t[\xbold_0 \lhd \xbold_1/\xbold] = t[\xbold_0/\xbold] \lhd t[\xbold_1/\xbold], \] and
\[ \T_\Quandle(\xbold_0, \xbold_1, \y_1, \ldots, \y_n) \vdash t[\xbold_0 \lhd^{-1} \xbold_1/\xbold] = t[\xbold_0/\xbold] \lhd^{-1} t[\xbold_1/\xbold]. \]

\begin{theorem}[\textbf{Isotropy Group of a Free Quandle}]
\label{quandleisotropy}
Let $\Q_n$ be the free quandle on $n$ generators $\y_1, \ldots, \y_n$. Then for any $t \in \Term^c(\Sigma(\xbold, \y_1, \ldots, \y_n))$, we have \[ [t] \in G_{\T_\Quandle}(\Q_n) \] iff there are $m \geq 0$ and $1 \leq i_1, \ldots, i_m \leq n$ and $\epsilon_1, \ldots, \epsilon_m = \pm 1$ such that 
\[ [t] = \left[\xbold \lhd^{\epsilon_1} \y_{i_1} \lhd^{\epsilon_2} \ldots \lhd^{\epsilon_m} \y_{i_m}\right] \] and the corresponding group word $\y_{i_1}^{\epsilon_1} \ldots \y_{i_m}^{\epsilon_m}$ is reduced (we have written $\xbold \lhd^{\epsilon_1} \y_{i_1} \lhd^{\epsilon_2} \ldots \lhd^{\epsilon_m} \y_{i_m}$ instead of the more cumbersome $(\ldots((\xbold \lhd^{\epsilon_1} \y_{i_1}) \lhd^{\epsilon_2} \y_{i_2}) \ldots ) \lhd^{\epsilon_m} \y_{i_m}$, i.e. bracketing of quandle terms is assumed to associate to the left).  
\end{theorem}

\begin{proof}
We first show that the $\sim_\Quandle$-class of any term of the described form belongs to $G_{\T_\Quandle}(\Q_n)$. So let \[ t := \xbold \lhd^{\epsilon_1} \y_{i_1} \lhd^{\epsilon_2} \ldots \lhd^{\epsilon_m} \y_{i_m} \] for some $m \geq 0$ and $1 \leq i_1, \ldots, i_m \leq n$ and $\epsilon_1, \ldots, \epsilon_m = \pm 1$. We must show that $[t]$ is invertible and commutes generically with the quandle operations. For invertibility, we show that the $\sim_\Quandle$-class of 
\[ t^{-1} := \xbold \lhd^{-\epsilon_m} y_{i_m} \lhd^{-\epsilon_{m-1}}\ldots \lhd^{-\epsilon_1} y_{i_1} \]
is the inverse of $[t]$. So we must show that $t[t^{-1}/\xbold] \sim_\Quandle \xbold$ and $t^{-1}[t/\xbold] \sim_\Quandle \xbold$. Since the two claims have analogous proofs, we will only prove the first. 

By Theorem 1, it suffices to show that $\E(t[t^{-1}/\xbold]) \sim \E(\xbold) \equiv \xbold$ in the free group on $\xbold, \y_1, \ldots, \y_n$. Using Lemma 1, we have
\begin{align*}
\E(t[t^{-1}/\xbold])	&\equiv \E(t)[\E\left(t^{-1}\right)/\xbold] \\
			&\sim \left(\y_{i_m}^{-\epsilon_m} \ldots \y_{i_1}^{-\epsilon_1} \xbold \y_{i_1}^{\epsilon_1} \ldots \y_{i_m}^{\epsilon_m}\right)[\E\left(t^{-1}\right)/\xbold] \\
			&\sim \left(\y_{i_m}^{-\epsilon_m} \ldots \y_{i_1}^{-\epsilon_1} \xbold \y_{i_1}^{\epsilon_1} \ldots \y_{i_m}^{\epsilon_m}\right)\left[\y_{i_1}^{\epsilon_1} \ldots \y_{i_m}^{\epsilon_m} \xbold \y_{i_m}^{-\epsilon_m} \ldots \y_{i_1}^{-\epsilon_1}/\xbold\right] \\
			&\equiv \y_{i_m}^{-\epsilon_m} \ldots \y_{i_1}^{-\epsilon_1}\y_{i_1}^{\epsilon_1} \ldots \y_{i_m}^{\epsilon_m}\xbold \y_{i_m}^{-\epsilon_m} \ldots \y_{i_1}^{-\epsilon_1}\y_{i_1}^{\epsilon_1} \ldots \y_{i_m}^{\epsilon_m} \\
			&\sim \xbold \\
			&\equiv \E(\xbold),
\end{align*}

\noindent as desired. 

Now we show that $[t]$ commutes generically with the quandle operations. Since the proofs of both claims are analogous, we only prove that $[t]$ commutes generically with $\lhd$. By Theorem 1, it suffices to show that 
\[ \E(t[\xbold_0 \lhd \xbold_1/\xbold]) \sim \E(t[\xbold_0/\xbold] \lhd t[\xbold_1/\xbold]) \] in the free group on $\xbold_0, \xbold_1, \y_1, \ldots, \y_n$. Starting from the right side, we have
\begin{align*}
&\ \ \ \ \E(t[\xbold_0/\xbold] \lhd t[\xbold_1/\xbold]) \\	
&= \E(t[\xbold_1/\xbold])^{-1} \cdot \E(t[\xbold_0/\xbold]) \cdot \E(t[\xbold_1/\xbold]) \\
&\sim \E(t)^{-1}[\xbold_1/\xbold] \cdot \E(t[\xbold_0/\xbold]) \cdot \E(t[\xbold_1/\xbold]) \\
&\sim \y_{i_m}^{-\epsilon_m} \ldots \y_{i_1}^{-\epsilon_1} \xbold_1^{-1} \y_{i_1}^{\epsilon_1} \ldots \y_{i_m}^{\epsilon_m} \cdot \y_{i_m}^{-\epsilon_m} \ldots \y_{i_1}^{-\epsilon_1} \xbold_0 \y_{i_1}^{\epsilon_1} \ldots \y_{i_m}^{\epsilon_m} \cdot \y_{i_m}^{-\epsilon_m} \ldots \y_{i_1}^{-\epsilon_1} \xbold_1 \y_{i_1}^{\epsilon_1} \ldots \y_{i_m}^{\epsilon_m} \\
&\sim \y_{i_m}^{-\epsilon_m} \ldots \y_{i_1}^{-\epsilon_1} \xbold_1^{-1}\xbold_0\xbold_1 \y_{i_1}^{\epsilon_1} \ldots \y_{i_m}^{\epsilon_m} \\
&\equiv \left(\y_{i_m}^{-\epsilon_m} \ldots \y_{i_1}^{-\epsilon_1} \xbold \y_{i_1}^{\epsilon_1} \ldots \y_{i_m}^{\epsilon_m}\right)[\xbold_1^{-1}\xbold_0\xbold_1/\xbold] \\
&\sim \E(t)[\xbold_1^{-1}\xbold_0\xbold_1/\xbold] \\
&= \E(t)[\E(\xbold_0 \lhd \xbold_1)/\xbold] \\
&\sim \E(t[\xbold_0 \lhd \xbold_1/\xbold]),
\end{align*}
as desired (where we applied Lemma 1 to obtain the final congruence). This completes the proof that $[t] \in G_{\T_\Quandle}(\Q_n)$.

Now suppose that $t \in \Term^c(\Sigma(\xbold, \y_1, \ldots, \y_n))$ and $[t] \in G_{\T_\Quandle}(\Q_n)$. We show that $t$ can be assumed to have the form described in the statement of the theorem. Since $[t] \in G_{\T_\Quandle}(\Q_n)$, we know that $[t]$ commutes generically with the quandle operations. In particular, we have
\[ \T_\Quandle(\xbold_0, \xbold_1, \y_1, \ldots, \y_n) \vdash t[\xbold_0 \lhd \xbold_1/\xbold] = t[\xbold_0/\xbold] \lhd t[\xbold_1/\xbold].\]
By Theorem 1, it then follows that the following relation holds in the free group on the same generators:
\[ \E(t[\xbold_0 \lhd \xbold_1/\xbold]) \sim \E(t[\xbold_0/\xbold] \lhd t[\xbold_1/\xbold]). \]
Then since we have 
\begin{align*}
\E(t[\xbold_0/\xbold] \lhd t[\xbold_1/\xbold])	&= \E(t[\xbold_1/\xbold])^{-1} \cdot \E(t[\xbold_0/\xbold]) \cdot \E(t[\xbold_1/\xbold]) \\
					&\sim \E(t)^{-1}[\xbold_1/\xbold] \cdot \E(t)[\xbold_0/\xbold] \cdot \E(t)[\xbold_1/\xbold], 
\end{align*}
and since (by Lemma 1) we have 
\begin{align*}
\E(t[\xbold_0 \lhd \xbold_1/\xbold])	&\sim \E(t)[\E(\xbold_0 \lhd \xbold_1)/\xbold] \\
					&= \E(t)[\xbold_1^{-1}\xbold_0\xbold_1/\xbold],
\end{align*}
it follows that we have
\[ \E(t)[\xbold_1^{-1}\xbold_0\xbold_1/\xbold] \sim \E(t)^{-1}[\xbold_1/\xbold] \cdot \E(t)[\xbold_0/\xbold] \cdot \E(t)[\xbold_1/\xbold]. \]
Now let $s$ be the unique reduced word congruent to $\E(t) \in \Term^c(\Sigma_\Grp(\xbold, \y_1, \ldots, \y_n))$, so that $\E(t) \sim s$. Then we obtain 
\[ s[\xbold_1^{-1}\xbold_0\xbold_1/\xbold] \sim s^{-1}[\xbold_1/\xbold] \cdot s[\xbold_0/\xbold] \cdot s[\xbold_1/\xbold], \]
which implies 
 \[ s[\xbold_1/\xbold] \cdot s[\xbold_1^{-1}\xbold_0\xbold_1/\xbold] \sim s[\xbold_0/\xbold] \cdot s[\xbold_1/\xbold]. \] Then by Lemma 2, since $s$ is reduced, it follows that $s$ has at most one occurrence of $\xbold$, which will have exponent $1$. Now we show that $s$ has \emph{at least} one, and hence \emph{exactly one}, occurrence of $\xbold$. Since $[t] \in G_{\T_\Quandle}(\Q_n)$, it is invertible, and hence there is some $t^{-1} \in \Term^c(\Sigma(\xbold, \y_1, \ldots, \y_n))$ such that \[ \T_\Quandle(\xbold, \y_1, \ldots, \y_n) \vdash t[t^{-1}/\xbold] = \xbold = t^{-1}[t/\xbold]. \] Then by Theorem 1 and Lemma 1, it follows that \[ \xbold \equiv \E(\xbold) \sim \E(t[t^{-1}/\xbold]) \sim \E(t)[\E(t^{-1})/\xbold]. \] Since $\E(t) \sim s$, it then follows that \[ s[\E(t^{-1})/\xbold] \sim \xbold. \] If $s$ did \emph{not} have at least one occurrence of $\xbold$, then we would have $s[\E(t^{-1})/\xbold] \equiv s$, so that we could deduce $s \sim \xbold$. But then since $s$ and $\xbold$ are reduced, this would imply that $s \equiv \xbold$, contradicting the assumption that $s$ has no occurrence of $\xbold$. So it follows that $s$ has at least one, and hence exactly one, occurrence of $\xbold$. So then $s \equiv t_1 \cdot \xbold \cdot t_2$ for some reduced (possibly empty) words $t_1, t_2 \in \Term^c(\Sigma_\Grp(\y_1, \ldots, \y_n))$. From \[ s[\xbold_1/\xbold] \cdot s[\xbold_1^{-1}\xbold_0\xbold_1/\xbold] \sim s[\xbold_0/\xbold] \cdot s[\xbold_1/\xbold] \] we then infer
\[ t_1 \cdot \xbold_1 \cdot t_2 \cdot t_1 \cdot \xbold_1^{-1} \xbold_0 \xbold_1 \cdot t_2 \sim t_1 \cdot \xbold_0 \cdot t_2 \cdot t_1 \cdot \xbold_1 \cdot t_2. \] So the reductions of both words are identical, which implies that $t_2 \cdot t_1 \sim e$, so that $t_1 \sim t_2^{-1}$, and hence $s \equiv t_1 \cdot \xbold \cdot t_2 \sim t_1 \cdot \xbold \cdot t_1^{-1}$. So we now have $\E(t) \sim s \sim t_1 \cdot \xbold \cdot t_1^{-1}$ for some reduced word $t_1 \in \Term^c(\Sigma_\Grp(\y_1, \ldots, \y_n))$.
 
Now let $t_1 \equiv \y_{i_1}^{\epsilon_1} \ldots \y_{i_m}^{\epsilon_m}$ for some $m \geq 0$, with $\epsilon_j = \pm 1$ and $1 \leq i_j \leq n$ for each $1 \leq j \leq m$. Then we have $t_1^{-1} \sim \y_{i_m}^{-\epsilon_m} \ldots \y_{i_1}^{-\epsilon_1}$, so that \[ \E(t) \sim \y_{i_1}^{\epsilon_1} \ldots \y_{i_m}^{\epsilon_m}\xbold \y_{i_m}^{-\epsilon_m} \ldots \y_{i_1}^{-\epsilon_1}. \] Then we have
 \begin{align*}
 \E(t)	&\sim \y_{i_1}^{\epsilon_1} \ldots \y_{i_m}^{\epsilon_m}\xbold \y_{i_m}^{-\epsilon_m} \ldots \y_{i_1}^{-\epsilon_1} \\
 		&\sim \E(\xbold \lhd^{-\epsilon_m} \y_{i_m} \lhd^{-\epsilon_{m-1}} \ldots \lhd^{-\epsilon_1} \y_{i_1}).	
\end{align*}
By Theorem 1, we then deduce that 
\[ t \ \sim_\Quandle \ \xbold \lhd^{-\epsilon_m} \y_{i_m} \lhd^{-\epsilon_{m-1}} \ldots \lhd^{-\epsilon_1} \y_{i_1}, \]
so that $t$ is $\sim_\Quandle$-congruent to a term of the desired form (since $t_1 \equiv \y_{i_1}^{\epsilon_1} \ldots \y_{i_m}^{\epsilon_m}$ is reduced, which implies that $\y_{i_m}^{-\epsilon_m} \ldots \y_{i_1}^{-\epsilon_1}$ is reduced). This completes the proof of the theorem.
\end{proof}

\noindent Given this logical description of the isotropy group of the free quandle on $n$ generators, we now give a more \emph{algebraic} description of this isotropy group:

\begin{corollary}
\label{quandleisotropycor}
Let $\F_n$ be the free group on $n$ generators $\y_1, \ldots, \y_n$. Then \[ G_{\T_\Quandle}(\Q_n) \cong \F_{n}. \] That is, the logical isotropy group of the free quandle on $n$ generators is isomorphic to the free group on $n$ generators. 
\end{corollary}

\begin{proof}
Since $\F_n$ is the free group on $n$ generators and $G_{\T_\Quandle}(\Q_n)$ is a group, there is a unique group homomorphism \[ \phi : \F_n \to G_{\T_\Quandle}(\Q_n) \] with \[ \phi([\y_i]) = [\xbold \lhd \y_i] \] for each $1 \leq i \leq n$, since $[\xbold \lhd \y_i] \in G_{\T_\Quandle}(\Q_n)$ by Theorem 2. So it remains to show that $\phi$ is a bijection. Note first that for any $1 \leq i \leq n$ we have 
\begin{align*}
\phi\left(\left[\y_i^{-1}\right]\right)	&= \phi\left(\left[\y_i\right]\right)^{-1} \\
			&= [\xbold \lhd \y_i]^{-1} \\
			&= [\xbold \lhd^{-1} \y_i]
\end{align*}
(cf. the proof of Theorem 2 for the last equality), and hence for any $m \geq 1$ with $\epsilon_j = \pm 1$ and $1 \leq i_j \leq n$ for each $1 \leq j \leq m$, we have (since the product in $G_{\T_\Quandle}(\Q_n)$ is given by substitution into $\xbold$)
\begin{align*}
\phi\left(\left[\y_{i_1}^{\epsilon_1} \ldots \y_{i_m}^{\epsilon_m}\right]\right)	&= \phi\left(\left[\y_{i_1}^{\epsilon_1}\right]\right) \cdot \ldots \cdot \phi\left(\left[\y_{i_m}^{\epsilon_m}\right]\right) \\
								&= [\xbold \lhd^{\epsilon_1} \y_{i_1}] \cdot \ldots \cdot [\xbold \lhd^{\epsilon_m} \y_{i_m}] \\
								&= [\xbold \lhd^{\epsilon_m} \y_{i_m} \lhd^{\epsilon_{m-1}} \ldots \lhd^{\epsilon_1} \y_{i_1}].
\end{align*}
Now, that $\phi$ is surjective is obvious, because by Theorem 2, if $[t] \in G_{\T_\Quandle}(\Q_n)$, then either $[t] = [\xbold]$, in which case we have $\phi([e]) = [\xbold] = [t]$ (because $\phi$ is a group homomorphism and $[\xbold]$ is the identity element of $G_{\T_\Quandle}(\Q_n)$), or otherwise there is some $m \geq 1$ such that
\[ [t] = [\xbold \lhd^{\epsilon_1} \y_{i_1} \lhd^{\epsilon_2} \ldots \lhd^{\epsilon_m} \y_{i_m}], \] with $\epsilon_j = \pm 1$ and $1 \leq i_j \leq n$ for all $1 \leq j \leq m$. But then we have
\[ \phi\left(\left[\y_{i_m}^{\epsilon_m} \ldots \y_{i_1}^{\epsilon_1}\right]\right) = [t], \] as desired. \par

To prove that $\phi$ is injective, let $\y_{i_1}^{\epsilon_1} \ldots \y_{i_m}^{\epsilon_m}$ and $\y_{j_1}^{\delta_1} \ldots \y_{j_p}^{\delta_p}$ be reduced group words over the generators $\y_1, \ldots, \y_n$ with $m, p \geq 1$ (if one of the words is just $e$, then the argument that follows is even easier), and suppose that 
\[ \phi\left(\left[\y_{i_1}^{\epsilon_1} \ldots \y_{i_m}^{\epsilon_m}\right]\right) = \phi\left(\left[\y_{j_1}^{\delta_1} \ldots \y_{j_p}^{\delta_p}\right]\right), \] in order to show that \[ \y_{i_1}^{\epsilon_1} \ldots \y_{i_m}^{\epsilon_m} \sim \y_{j_1}^{\delta_1} \ldots \y_{j_p}^{\delta_p} \] in the free group on $y_1, \ldots, y_n$. The assumption implies that 
\[ \xbold \lhd^{\epsilon_m} y_{i_m} \lhd^{\epsilon_{m-1}} \ldots \lhd^{\epsilon_1} y_{i_1} \ \sim_\Quandle \ \xbold \lhd^{\delta_p} y_{j_p} \lhd^{\delta_{p-1}} \ldots \lhd^{\delta_1} y_{j_1} \] 
in the free quandle on $\xbold, \y_1, \ldots, \y_n$. By Theorem 1, this in turn implies that 
\[ \E\left(\xbold \lhd^{\epsilon_m} \y_{i_m} \lhd^{\epsilon_{m-1}} \ldots \lhd^{\epsilon_1} \y_{i_1}\right) \sim \E\left(\xbold \lhd^{\delta_p} \y_{j_p} \lhd^{\delta_{p-1}} \ldots \lhd^{\delta_1} \y_{j_1}\right) \] in the free group on $\xbold, \y_1, \ldots, \y_n$, i.e. that
\[ \y_{i_1}^{-\epsilon_1} \ldots \y_{i_m}^{-\epsilon_m}\xbold \y_{i_m}^{\epsilon_m} \ldots \y_{i_1}^{\epsilon_1} \sim 
\y_{j_1}^{-\delta_1} \ldots \y_{j_p}^{-\delta_p}\xbold \y_{j_p}^{\delta_p} \ldots \y_{j_1}^{\delta_1} \]
in the free group on $\xbold, \y_1, \ldots, \y_n$. This implies that 
\[ \xbold \sim \y_{i_m}^{\epsilon_m} \ldots \y_{i_1}^{\epsilon_1}\y_{j_1}^{-\delta_1} \ldots \y_{j_p}^{-\delta_p}\xbold \y_{j_p}^{\delta_p} \ldots \y_{j_1}^{\delta_1} \y_{i_1}^{-\epsilon_1} \ldots \y_{i_m}^{-\epsilon_m}, \]
which then implies that
\[ \y_{i_m}^{\epsilon_m} \ldots \y_{i_1}^{\epsilon_1} \y_{j_1}^{-\delta_1} \ldots \y_{j_p}^{-\delta_p} \sim e \]
and
\[ \y_{j_p}^{\delta_p} \ldots \y_{j_1}^{\delta_1} \y_{i_1}^{-\epsilon_1} \ldots \y_{i_m}^{-\epsilon_m} \sim e, \] which finally imply that
\[ \y_{i_m}^{\epsilon_m} \ldots \y_{i_1}^{\epsilon_1} \sim \y_{j_p}^{\delta_p} \ldots \y_{j_1}^{\delta_1}. \] 
Since $\y_{i_1}^{\epsilon_1} \ldots \y_{i_m}^{\epsilon_m}$ and $\y_{j_1}^{\delta_1} \ldots \y_{j_p}^{\delta_p}$ are reduced words by assumption, this entails that $\y_{i_m}^{\epsilon_m} \ldots \y_{i_1}^{\epsilon_1}$ and $\y_{j_p}^{\delta_p} \ldots \y_{j_1}^{\delta_1}$ are also reduced words. Therefore, since we are working in the free group on $\y_1, \ldots, \y_n$, this implies that $m = p$ and $\y_{i_k} = \y_{j_k}$ and $\epsilon_k = \delta_k$ for all $1 \leq k \leq m = p$. This proves that $\y_{i_1}^{\epsilon_1} \ldots \y_{i_m}^{\epsilon_m} \sim \y_{j_1}^{\delta_1} \ldots \y_{j_p}^{\delta_p}$, as desired. 
\end{proof}

From our characterization(s) of the logical isotropy groups of the free, finitely generated quandles, we can now deduce characterizations of the \emph{categorical} isotropy groups of these quandles. The proof of the following corollary invokes the two bullet points preceding Definition 1 in Section 1, the characterization given in Theorem \ref{quandleisotropy}, and the fact that reduced group words congruent modulo $\sim_\Grp$ must be identical. 

\begin{corollary}
\label{quandlesisotropycor}
Let $n \geq 0$. 
\begin{enumerate}
\item Let \[ \pi = \left(\pi_h : \cod(h) \to \cod(h)\right)_{\dom(h) = \Q_n} \] be a (not necessarily natural) family of endomorphisms of quandles, indexed by quandle morphisms $h$ with domain $\Q_n$. Then $\pi \in \Z_{\T_\Quandle}(\Q_n)$ iff there is a unique reduced word
\[ \y_{i_1}^{\epsilon_1} \ldots \y_{i_m}^{\epsilon_m} \in \Term^c(\Sigma_\Grp(\y_1, \ldots, \y_n)) \] with the property that for any quandle morphism $h : \Q_n \to Q$ we have
\[ \pi_h(q) = q \lhd^{\epsilon_1} h_{i_1} \lhd^{\epsilon_2} \ldots \lhd^{\epsilon_m} h_{i_m} \in Q. \] 

\item Let $h : \Q_n \to \Q_n$ be a quandle endomorphism. Then $h$ is a categorical inner automorphism iff there is a unique reduced word $\y_{i_1}^{\epsilon_1} \ldots \y_{i_m}^{\epsilon_m} \in \Term^c(\Sigma_\Grp(\y_1, \ldots, \y_n))$ such that 
\[ h([s]) = \left[s \lhd^{\epsilon_1} \y_{i_1} \lhd^{\epsilon_2} \ldots \lhd^{\epsilon_m} \y_{i_m} \right] \in \Q_n \] for any $[s] \in \Q_n$ (so $s \in \Term^c(\Sigma(\y_1, \ldots, \y_n))$). 

\item Let $h : \Q_n \to \Q_n$ be a quandle endomorphism. Then $h$ is a \textbf{categorical} inner automorphism iff $h$ is an \textbf{algebraic} inner automorphism.   
\end{enumerate} \qed
\end{corollary}

\noindent Finally, we can deduce a characterization of the \emph{global isotropy group} of the category $\Quandle$ of quandles and their homomorphisms, i.e. the group $\Aut\left(\Id_\Quandle\right)$ of automorphisms of the identity functor $\Id_\Quandle : \Quandle \to \Quandle$ (which is also the group of invertible elements of the \emph{centre} of the category $\Quandle$, which is the monoid $\End\left(\Id_\Quandle\right)$ of natural \emph{endo}morphisms of the identity functor). Since the category $\Quandle$ has an initial object, namely the absolutely free quandle $\Q_0$ (whose carrier is just the empty set), it is easy to see that the global isotropy group of $\Quandle$ is exactly the (covariant) categorical isotropy group of the initial object $\Q_0$, i.e.
\[ \Aut\left(\Id_\Quandle\right) = \mathcal{Z}_{\T_\Quandle}(\Q_0). \] Since
\[ \mathcal{Z}_{\T_\Quandle}(\Q_0) \cong G_{\T_\Quandle}(\Q_0) \cong \F_0 \] by Corollary \ref{quandleisotropycor} and $\F_0$ is the trivial group (being the free group on $0$ generators), we thus obtain:

\begin{corollary}
\label{globalquandlecor}
The global isotropy group of the category $\Quandle$ is the trivial group, i.e. the only automorphism of the identity functor $\Id_\Quandle$ is the identity natural transformation. \qed 
\end{corollary}

\noindent We also note in connection with Corollary \ref{globalquandlecor} that M. Szymik independently proved in \cite[Theorem 5.5]{Szymik} that the center $\End\left(\Id_\Quandle\right)$ of the category $\Quandle$ is trivial as well. Thus, we obtain the following further corollary:

\begin{corollary}
The global isotropy group of the category $\Quandle$ is equal to its center, and both are trivial. \qed
\end{corollary}

\section{Isotropy Groups of Free Racks}

In this section, we will proceed to characterize the isotropy groups of free, finitely generated racks, which is a slightly more involved task than the characterization for quandles (due to the increased complexity of the word problem for free racks). Given a (finite) set $X$, it has been shown (cf. \cite[Proposition 4.2]{SDworld}) that the free rack on $X$ has the following presentation, which we denote by $\HalfConj(X, \F_X)$, where $\F_X$ is once again the free group on $X$. The underlying set of $\HalfConj(X, \F_X)$ is the set $X \times \F_X$, and the rack operations on this set are defined as follows, for any $\xbold, \y \in X$ and $[s], [t] \in \F_X$:

\[ (\xbold, [s]) \lhd (\y, [t]) := (\xbold, [s \cdot t^{-1} \cdot \y \cdot t]), \]

\[ (\xbold, [s]) \lhd^{-1} (\y, [t]) := (\xbold, [s \cdot t^{-1} \cdot \y^{-1} \cdot t]). \]

\noindent There is now a function \[ \E_X : \Term^c(\Sigma(X)) \to X \times \Term^c(\Sigma_\Grp(X)) \] with \[ \E_X(\xbold) := (\xbold, e) \tag{$\xbold \in X$} \] and
\[ \E_X(s \lhd^\epsilon t) := \left(\pi_1(\E_X(s)), \pi_2(\E_X(s)) \cdot \pi_2(\E_X(t))^{-1} \cdot \pi_1(\E_X(t))^\epsilon \cdot \pi_2(\E_X(t))\right) \] for $\epsilon = \pm 1$ and $s, t \in \Term^c(\Sigma(X))$.
As before, we will omit the subscript on $\E$ to increase readability.

First, we have the following definition and lemma concerning the relationship between $\E$ and the first projection function $\pi_1 : \HalfConj(X, \F_X) \to X$.

\begin{definition}
{\em Let $t \in \Term^c(\Sigma(X))$ for a (finite) set $X$. We define $\Left(t) \in X$ (intuitively, the `leftmost' element of $X$ occurring in $t$) by induction on $t$:

\begin{itemize}

\item If $t \equiv \xbold$ for some $\xbold \in X$, then $\Left(t) := \xbold$. 

\item If $t \equiv t_1 \lhd^{\epsilon} t_2$ for $t_1, t_2 \in \Term^c(\Sigma(X))$ and $\epsilon = \pm 1$, then $\Left(t) := \Left(t_1)$.
\end{itemize} \qed
}
\end{definition}

\noindent To increase readability, we will now write $\pi_i(\xbold, s)$ as $(\xbold, s)_i$ for $i \in \{1, 2\}$ and $(x, s) \in X \times \Term^c(\Sigma_\Grp(X))$.  

\begin{lemma}
\label{leftlemma}
Let $t \in \Term^c(\Sigma(X))$ for a (finite) set $X$. Then 
\[ \E(t)_1 = \Left(t) \in X. \]
\end{lemma}

\begin{proof}
We prove this by induction on $t$. 
\begin{itemize}
\item If $t \equiv \xbold$ for some $\xbold \in X$, then we have
\[ \E(t)_1 = \E(\xbold)_1 = (\xbold, e)_1 = \xbold = \Left(\xbold) = \Left(t), \]
as desired.

\item Let $t \equiv t_1 \lhd^{\epsilon} t_2$ for some $t_1, t_2 \in \Term^c(\Sigma(X))$ and $\epsilon = \pm 1$ such that $\E(t_1)_1 = \Left(t_1)$. Then by definition of $\E$ we have
\[ \E(t_1 \lhd^{\epsilon} t_2)_1 = \E(t_1)_1 = \Left(t_1) = \Left(t_1 \lhd^\epsilon t_2), \]
as desired.  
\end{itemize}
\end{proof}

\noindent We also have the following substitution lemma:

\begin{lemma}
\label{racksubstlemma}
Let $X$ be an arbitrary (finite) set with designated element $\xbold \in X$, and let $X'$ be another (finite) set with $X \subseteq X'$ and $\xbold_0, \xbold_1 \in X' \setminus X$. Then for any $t \in \Term^c(\Sigma(X))$ we have:
\begin{itemize}

\item If $\E(t)_1 = \xbold$, then \[ \E(t[\xbold_0 \lhd^\epsilon \xbold_1/\xbold])_1 = \xbold_0 \] for $\epsilon = \pm 1$ and
\[ \E(t[\xbold_0 \lhd \xbold_1/\xbold])_2 \sim \xbold_1 \cdot \E(t)_2[\xbold_1^{-1}\xbold_0\xbold_1/\xbold], \]
\[ \E(t[\xbold_0 \lhd^{-1} \xbold_1/\xbold])_2 \sim \xbold_1^{-1} \cdot \E(t)_2[\xbold_1\xbold_0\xbold_1^{-1}/\xbold]. \]

\item If $\E(t)_1 \neq \xbold$, then \[ \E(t[\xbold_0 \lhd^\epsilon \xbold_1/\xbold])_1 = \E(t)_1 \] for $\epsilon = \pm 1$ and
\[ \E(t[\xbold_0 \lhd \xbold_1/\xbold])_2 \sim \E(t)_2[\xbold_1^{-1}\xbold_0\xbold_1/\xbold], \]
\[ \E(t[\xbold_0 \lhd^{-1} \xbold_1/\xbold])_2 \sim \E(t)_2[\xbold_1\xbold_0\xbold_1^{-1}/\xbold]. \]
\end{itemize}
\end{lemma}

\begin{proof}
We prove this by induction on $t \in \Term^c(\Sigma(X))$. We will only consider the claims for $\lhd$, since the claims for $\lhd^{-1}$ have analogous proofs. 
\begin{itemize}
\item If $t \equiv \xbold$, then $\E(t)_1 = (\xbold, e)_1 = \xbold$, so that we must prove  
\[ \E(\xbold_0 \lhd \xbold_1)_1 = \xbold_0 \] and
\[ \E(\xbold_0 \lhd \xbold_1)_2 \sim \xbold_1 \cdot e[\xbold_1^{-1}\xbold_0\xbold_1/\xbold] \equiv \xbold_1 \cdot e. \] 
Since $\E(\xbold_0) = (\xbold_0, e)$ and $\E(\xbold_1) = (\xbold_1, e)$, we have by definition of $\E$
\[ \E(\xbold_0 \lhd \xbold_1) = (\xbold_0, e \cdot e^{-1} \cdot \xbold_1 \cdot e), \]
which clearly yields the desired result. 

\item If $t \equiv \y$ for some $\y \in X$ with $\y \neq \xbold$, then $\E(t)_1 = (\y, e)_1 = \y$ and hence we must show
\[ \E(\y)_1 = \y \] and 
\[ \E(\y)_2 \sim e[\xbold_1^{-1}\xbold_0\xbold_1/\xbold]) \equiv e, \]
which clearly follows by definition of $\E(\y)$.

\item For the induction step, suppose that $t \equiv t_1 \lhd t_2$ for some $t_1, t_2 \in \Term^c(\Sigma(X))$ for which the result holds.
\begin{itemize}
\item Suppose first that $\E(t)_1 = \xbold$. Then by Lemma \ref{leftlemma} it follows that \[ \xbold = \Left(t) = \Left(t_1) = \E(t_1)_1. \]
Then by the induction hypothesis for $t_1$ we have
\[ \E(t_1[\xbold_0 \lhd \xbold_1/\xbold])_1 = \xbold_0 \] and
\[ \E(t_1[\xbold_0 \lhd \xbold_1/\xbold])_2 \sim \xbold_1 \cdot \E(t_1)_2[\xbold_1^{-1}\xbold_0\xbold_1/\xbold], \] so that by definition of $\E$ we have
\[ \E(t[\xbold_0 \lhd \xbold_1/\xbold])_1 = \E(t_1[\xbold_0 \lhd \xbold_1/\xbold])_1 = \xbold_0, \] as required. 
To compute $\E(t[\xbold_0 \lhd \xbold_1/\xbold])_2$, suppose in addition that $\E(t_2)_1 = \xbold$, so that the induction hypothesis for $t_2$ gives
\[ \E(t_2[\xbold_0 \lhd \xbold_1/\xbold])_1 = \xbold_0 \] and
\[ \E(t_2[\xbold_0 \lhd \xbold_1/\xbold])_2 \sim \xbold_1 \cdot \E(t_2)_2[\xbold_1^{-1}\xbold_0\xbold_1/\xbold]. \]
Using the definition of $\E$, the induction hypotheses, and the current assumption that $\E(t_2)_1 = \xbold$, we then have:
\begin{align*}
&\quad \ \E(t[\xbold_0 \lhd \xbold_1/\xbold])_2 \\	
&= \E(t_1[\xbold_0 \lhd \xbold_1/\xbold] \lhd t_2[\xbold_0 \lhd \xbold_1/\xbold])_2 \\
&\sim \xbold_1 \cdot \E(t_1)_2[\xbold_1^{-1}\xbold_0\xbold_1/\xbold] \cdot \left(\E(t_2)_2[\xbold_1^{-1}\xbold_0\xbold_1/\xbold]\right)^{-1} \cdot \xbold_1^{-1} \cdot \xbold_0 \cdot \xbold_1 \cdot\E(t_2)_2[\xbold_1^{-1}\xbold_0\xbold_1/\xbold] \\
&\sim \xbold_1 \cdot \E(t_1)_2[\xbold_1^{-1}\xbold_0\xbold_1/\xbold] \cdot \E(t_2)_2^{-1}[\xbold_1^{-1}\xbold_0\xbold_1/\xbold] \cdot \xbold_1^{-1} \cdot \xbold_0 \cdot \xbold_1 \cdot \E(t_2)_2[\xbold_1^{-1}\xbold_0\xbold_1/\xbold] \\
&\equiv \xbold_1 \cdot \left(\E(t_1)_2 \cdot \E(t_2)_2^{-1} \cdot \xbold \cdot \E(t_2)_2\right)[\xbold_1^{-1}\xbold_0\xbold_1/\xbold] \\
&= \xbold_1 \cdot \left(\E(t_1)_2 \cdot \E(t_2)_2^{-1} \cdot \E(t_2)_1 \cdot \E(t_2)_2\right)[\xbold_1^{-1}\xbold_0\xbold_1/\xbold] \\
&= \xbold_1 \cdot \E(t_1 \lhd t_2)_2[\xbold_1^{-1}\xbold_0\xbold_1/\xbold] \\
&= \xbold_1 \cdot \E(t)_2[\xbold_1^{-1}\xbold_0\xbold_1/\xbold],
\end{align*}
as desired. 

Now suppose that $\E(t_2)_1 = \y$ for some $\y \in X$ with $\y \neq \xbold$. Then by the induction hypothesis for $t_2$, it follows that \[ \E(t_2[\xbold_0 \lhd \xbold_1/\xbold])_1 = \y \] and 
\[ \E(t_2[\xbold_0 \lhd \xbold_1/\xbold])_2 \sim \E(t_2)_2[\xbold_1^{-1}\xbold_0\xbold_1/\xbold]. \]
Then we calculate as follows:
\begin{align*}
&\quad \ \E(t[\xbold_0 \lhd \xbold_1/\xbold])_2 \\	
&= \E(t_1[\xbold_0 \lhd \xbold_1/\xbold] \lhd t_2[\xbold_0 \lhd \xbold_1/\xbold])_2 \\
&\sim \xbold_1 \cdot \E(t_1)_2[\xbold_1^{-1}\xbold_0\xbold_1/\xbold] \cdot \left(\E(t_2)_2[\xbold_1^{-1}\xbold_0\xbold_1/\xbold]\right)^{-1} \cdot \y \cdot \E(t_2)_2[\xbold_1^{-1}\xbold_0\xbold_1/\xbold] \\
&\sim \xbold_1 \cdot \E(t_1)_2[\xbold_1^{-1}\xbold_0\xbold_1/\xbold] \cdot \E(t_2)_2^{-1}[\xbold_1^{-1}\xbold_0\xbold_1/\xbold] \cdot \y \cdot \E(t_2)_2[\xbold_1^{-1}\xbold_0\xbold_1/\xbold] \\
&\equiv \xbold_1 \cdot \left(\E(t_1)_2 \cdot \E(t_2)_2^{-1} \cdot \y \cdot \E(t_2)_2\right)[\xbold_1^{-1}\xbold_0\xbold_1/\xbold] \\
&= \xbold_1 \cdot \left(\E(t_1)_2 \cdot \E(t_2)_2^{-1} \cdot \E(t_2)_1 \cdot \E(t_2)_2\right)[\xbold_1^{-1}\xbold_0\xbold_1/\xbold] \\
&= \xbold_1 \cdot \E(t_1 \lhd t_2)_2[\xbold_1^{-1}\xbold_0\xbold_1/\xbold] \\
&= \xbold_1 \cdot \E(t)_2[\xbold_1^{-1}\xbold_0\xbold_1/\xbold],
\end{align*}
as desired. This completes the proof for the case where $\E(t)_1 = \xbold$.

\item Now suppose that $\E(t)_1 = \y$ for some $\y \neq \xbold$. As before, this implies that $\E(t_1)_1 = \y$ as well. Then by the induction hypothesis for $t_1$, we have
\[ \E(t_1[\xbold_0 \lhd \xbold_1/\xbold])_1 = \y \] and
\[ \E(t_1[\xbold_0 \lhd \xbold_1/\xbold])_2 \sim \E(t_1)_2[\xbold_1^{-1}\xbold_0\xbold_1/\xbold], \] which implies as before that \[ \E(t[\xbold_0 \lhd \xbold_1/\xbold])_1 = \y. \] 
To compute $\E(t[\xbold_0 \lhd \xbold_1/\xbold])_2$, suppose first that $\E(t_2)_1 = \xbold$. Then by the induction hypothesis for $t_2$, we have 
\[ \E(t_2[\xbold_0 \lhd \xbold_1/\xbold])_1 = \xbold_0 \] and
\[ \E(t_2[\xbold_0 \lhd \xbold_1/\xbold])_2 \sim \xbold_1 \cdot\E(t_2)_2[\xbold_1^{-1}\xbold_0\xbold_1/\xbold]. \]
Then we calculate as follows:
\begin{align*}
&\quad \ \E(t[\xbold_0 \lhd \xbold_1/\xbold])_2 \\	
&= \E(t_1[\xbold_0 \lhd \xbold_1/\xbold] \lhd t_2[\xbold_0 \lhd \xbold_1/\xbold])_2 \\
&\sim \E(t_1)_2[\xbold_1^{-1}\xbold_0\xbold_1/\xbold] \cdot \left(\E(t_2)_2[\xbold_1^{-1}\xbold_0\xbold_1/\xbold]\right)^{-1} \cdot \xbold_1^{-1} \cdot \xbold_0 \cdot \xbold_1 \cdot\E(t_2)_2[\xbold_1^{-1}\xbold_0\xbold_1/\xbold] \\
&\sim \E(t_1)_2[\xbold_1^{-1}\xbold_0\xbold_1/\xbold] \cdot \E(t_2)_2^{-1}[\xbold_1^{-1}\xbold_0\xbold_1/\xbold] \cdot \xbold_1^{-1} \cdot \xbold_0 \cdot \xbold_1 \cdot \E(t_2)_2[\xbold_1^{-1}\xbold_0\xbold_1/\xbold] \\
&\equiv \left(\E(t_1)_2 \cdot \E(t_2)_2^{-1} \cdot \xbold \cdot \E(t_2)_2\right)[\xbold_1^{-1}\xbold_0\xbold_1/\xbold] \\
&= \left(\E(t_1)_2 \cdot \E(t_2)_2^{-1} \cdot \E(t_2)_1 \cdot \E(t_2)_2\right)[\xbold_1^{-1}\xbold_0\xbold_1/\xbold] \\
&= \E(t_1 \lhd t_2)_2[\xbold_1^{-1}\xbold_0\xbold_1/\xbold] \\
&= \E(t)_2[\xbold_1^{-1}\xbold_0\xbold_1/\xbold],
\end{align*}
as desired. 

Finally, suppose that $\E(t_2)_1 = \z$ for some $\z \in X$ with $\z \neq \xbold$. Then by the induction hypothesis for $t_2$, we have
\[ \E(t_2[\xbold_0 \lhd \xbold_1/\xbold])_1 = \z \] and
\[ \E(t_2[\xbold_0 \lhd \xbold_1/\xbold])_2 \sim \E(t_2)_2[\xbold_1^{-1}\xbold_0\xbold_1/\xbold]). \]
Then we calculate as follows:
\begin{align*}
&\quad \ \E(t[\xbold_0 \lhd \xbold_1/\xbold])_2 \\	
&= \E(t_1[\xbold_0 \lhd \xbold_1/\xbold] \lhd t_2[\xbold_0 \lhd \xbold_1/\xbold])_2 \\
&\sim \E(t_1)_2[\xbold_1^{-1}\xbold_0\xbold_1/\xbold] \cdot \left(\E(t_2)_2[\xbold_1^{-1}\xbold_0\xbold_1/\xbold]\right)^{-1} \cdot \z \cdot \E(t_2)_2[\xbold_1^{-1}\xbold_0\xbold_1/\xbold] \\
&\sim \E(t_1)_2[\xbold_1^{-1}\xbold_0\xbold_1/\xbold] \cdot \E(t_2)_2^{-1}[\xbold_1^{-1}\xbold_0\xbold_1/\xbold] \cdot \z \cdot \E(t_2)_2[\xbold_1^{-1}\xbold_0\xbold_1/\xbold] \\
&\equiv \left(\E(t_1)_2 \cdot \E(t_2)_2^{-1} \cdot \z \cdot \E(t_2)_2\right)[\xbold_1^{-1}\xbold_0\xbold_1/\xbold] \\
&= \left(\E(t_1)_2 \cdot \E(t_2)_2^{-1} \cdot \E(t_2)_1 \cdot \E(t_2)_2\right)[\xbold_1^{-1}\xbold_0\xbold_1/\xbold] \\
&= \E(t_1 \lhd t_2)_2[\xbold_1^{-1}\xbold_0\xbold_1/\xbold] \\
&= \E(t)_2[\xbold_1^{-1}\xbold_0\xbold_1/\xbold],
\end{align*}
as desired. 
\end{itemize}
This completes the proof for the case $t \equiv t_1 \lhd t_2$, which completes the induction and hence the proof.
\end{itemize}
\end{proof}

\noindent We now require the following lemma, definition, and lemma.

\begin{lemma}
\label{canonicalformlemma}
Let $t \in \Term^c(\Sigma(X))$ for a (finite) set $X$, and assume that $t$ has the form
\[ t \equiv \z_0 \lhd^{\epsilon_1} \ldots \lhd^{\epsilon_m} \z_m, \]
with $\epsilon_j = \pm 1$ and $\z_k \in X$ for all $1 \leq j \leq m$ and $0 \leq k \leq m$. Then \[ \E(t)_1 = \z_0 \] and 
\[ \E(t)_2 \sim \z_1^{\epsilon_1} \cdot \ldots \cdot \z_m^{\epsilon_m}. \]
\end{lemma}

\begin{proof}
We prove this by induction on $m \geq 0$. If $m = 0$, then $t \equiv \z_0$ for some $\z_0 \in X$, and we must show that $\E(\z_0)_1 = \z_0$ and $\E(\z_0)_2 \sim e$, which is true by definition of $\E$. 

Now suppose that the result holds for some $m \geq 0$, and let 
\[ t \equiv \z_0 \lhd^{\epsilon_1} \ldots \lhd^{\epsilon_m} \z_m \lhd^{\epsilon_{m+1}} \z_{m+1}, \]
and let
\[ s := \z_0 \lhd^{\epsilon_1} \ldots \lhd^{\epsilon_{m}} \z_{m}. \]
Then by the induction hypothesis for $s$, we have
\[ \E(s)_1 = \z_0 \] and
\[ \E(s)_2 \sim \z_1^{\epsilon_1} \cdot \ldots \cdot \z_{m}^{\epsilon_{m}}. \]
By definition of $\E$ and the induction hypothesis (and the fact that bracketing associates to the left), we then have
\[ \E(t)_1 = \E(\z_0 \lhd^{\epsilon_1} \ldots \lhd^{\epsilon_m} \z_m \lhd^{\epsilon_{m+1}} \z_{m+1})_1 = \E(s)_1 = \z_0, \]  
as well as (recalling that $\E(\z_{m+1}) = (\z_{m+1}, e)$)
\begin{align*}
\E(t)_2	&= \E(\z_0 \lhd^{\epsilon_1} \ldots \lhd^{\epsilon_m} \z_m \lhd^{\epsilon_{m+1}} \z_{m+1})_2 \\
		&= \E(s \lhd^{\epsilon_{m+1}} \z_{m+1})_2 \\
		&\sim \z_1^{\epsilon_1} \cdot \ldots \cdot \z_{m}^{\epsilon_{m}} \cdot e^{-1} \cdot \z_{m+1}^{\epsilon_{m+1}} \cdot e \\
		&\sim \z_1^{\epsilon_1} \cdot \ldots \cdot \z_{m}^{\epsilon_{m}} \cdot \z_{m+1}^{\epsilon_{m+1}},
\end{align*}
as desired. This completes the induction and hence the proof.
\end{proof}

\begin{definition}
Let $t \in \Term^c(\Sigma(X))$ for a (finite) set $X$. Then we define $\W(t) \in \Term^c(\Sigma_\Grp(X))$ to be
\[ \W(t) := \E(t)_2^{-1} \cdot \E(t)_1 \cdot \E(t)_2. \] \qed
\end{definition}

\begin{lemma}
\label{mainracksubstlemma}
Let $t, t' \in \Term^c(\Sigma(X))$ for a (finite) set $X$, where $\xbold \in X$ is a distinguished element and $\E(t')_1 = \xbold$ and $t$ has the form
\[ t \equiv \xbold \lhd^{\epsilon_1} \z_1 \lhd^{\epsilon_2} \ldots \lhd^{\epsilon_m} \z_m, \]
with $\epsilon_j = \pm 1$ and $\z_j \in X$ for all $1 \leq j \leq m$. Then 
\[ \E(t[t'/\xbold])_1 = \xbold \] and
\[ \E(t[t'/\xbold])_2 \sim \E(t')_2 \cdot \E(t)_2[\W(t')/\xbold]. \]
\end{lemma}

\begin{proof}
We prove this by induction on the length of $t$. 
\begin{itemize}
\item For the base case, let $t \equiv \xbold$, so that $\E(t)_2 = \E(\xbold)_2 = \pi_2(\xbold, e) = e$. Then
\[ \E(t[t'/\xbold])_1 = \E(\xbold[t'/\xbold])_1 = \E(t')_1 = \xbold \] by hypothesis on $t'$, and we have
\begin{align*}
\E(t[t'/\xbold])_2	&= \E(\xbold[t'/\xbold])_2 \\
			&= \E(t')_2 \\
			&\sim \E(t')_2 \cdot e \\
			&\equiv \E(t')_2 \cdot e[\W(t')/\xbold] \\
			&= \E(t')_2 \cdot \E(t)_2[\W(t')/\xbold],
\end{align*}
as required. \par

\item Now suppose that the result holds for all terms $t$ of the described form of some length $n \geq 1$, and consider 
\[ t \equiv \xbold \lhd^{\epsilon_1} \z_1 \lhd^{\epsilon_2} \ldots \lhd^{\epsilon_{m+1}} z_{m+1}, \]
with $m \geq 0$ and $\epsilon_j = \pm 1$ and $\z_j \in X$ for all $1 \leq j \leq m + 1$. If we set
\[ s := \xbold \lhd^{\epsilon_1} \z_1 \lhd^{\epsilon_2} \ldots \lhd^{\epsilon_{m}} \z_{m}, \]
then by the induction hypothesis we have
\[ \E(s[t'/\xbold])_1 = \xbold \] and
\[ \E(s[t'/\xbold])_2 \sim \E(t')_2 \cdot \E(s)_2[\W(t')/\xbold]. \]
So $t \equiv s \lhd^{\epsilon_{m + 1}} \z_{m + 1}$, and hence we have
\[ \E(t[t'/\xbold])_1 = \E\left(s[t'/\xbold] \lhd^{\epsilon_{m+1}} \z_{m+1}[t'/\xbold]\right)_1 = \E(s[t'/\xbold])_1 = \xbold \] by definition of $\E$ and the induction hypothesis, as well as
\begin{align*}
&\quad \ \E(t[t'/\xbold])_2 \\	
&= \E(s[t'/\xbold] \lhd^{\epsilon_{m+1}} \z_{m+1}[t'/\xbold])_2 \\
&\sim \E(t')_2 \cdot \E(s)_2[\W(t')/\xbold] \cdot \E(z_{m+1}[t'/\xbold])_2^{-1} \cdot \E(\z_{m+1}[t'/\xbold])_1^{\epsilon_{m+1}} \cdot \E(\z_{m+1}[t'/\xbold])_2 \\
&\sim \E(t')_2 \cdot \left(\xbold \cdot \z_1^{\epsilon_1} \cdot \ldots \cdot \z_m^{\epsilon_m} \right)[\W(t')/\xbold] \cdot \E(\z_{m+1}[t'/\xbold])_2^{-1} \cdot \E(\z_{m+1}[t'/\xbold])_1^{\epsilon_{m+1}} \cdot \E(\z_{m+1}[t'/\xbold])_2 
\end{align*}
with the last congruence justified by Lemma \ref{canonicalformlemma}. 

Suppose first that $\z_{m+1} \neq \xbold$. Then we have $\E(\z_{m+1}[t'/\xbold]) = \E(\z_{m+1}) = (\z_{m+1}, e)$, so it follows that
\begin{align*}
&\quad \ \E(t[t'/\xbold])_2 \\	
&\sim \E(t')_2 \cdot \left(\xbold \cdot \z_1^{\epsilon_1} \cdot \ldots \cdot \z_m^{\epsilon_m} \right)[\W(t')/\xbold] \cdot \E(\z_{m+1}[t'/\xbold])_2^{-1} \cdot \E(\z_{m+1}[t'/\xbold])_1^{\epsilon_{m+1}} \cdot \E(\z_{m+1}[t'/\xbold])_2 \\
&\equiv \E(t')_2 \cdot \left(\xbold \cdot \z_1^{\epsilon_1} \cdot \ldots \cdot \z_m^{\epsilon_m} \right)[\W(t')/\xbold] \cdot e^{-1} \cdot \z_{m+1}^{\epsilon_{m+1}} \cdot e \\
&\sim \E(t')_2 \cdot \left(\xbold \cdot \z_1^{\epsilon_1} \cdot \ldots \cdot \z_m^{\epsilon_m} \right)[\W(t')/\xbold] \cdot \z_{m+1}^{\epsilon_{m+1}} \\
&\equiv \E(t')_2 \cdot \left(\xbold \cdot \z_1^{\epsilon_1} \cdot \ldots \cdot \z_{m+1}^{\epsilon_{m+1}} \right)[\W(t')/\xbold] \\
&\sim \E(t')_2 \cdot \E(t)_2[\W(t')/\xbold]   
\end{align*}
as desired, with the last equality justified by Lemma \ref{canonicalformlemma} and the assumption that $\z_{m+1} \neq \xbold$.  \par

Now suppose that $\z_{m+1} = \xbold$. Then we have $\E(\z_{m+1}[t'/\xbold]) = \E(\xbold[t'/\xbold]) = \E(t')$, and so we have
\begin{align*}
&\quad \ \E(t[t'/\xbold])_2 \\	
&\sim \E(t')_2 \cdot \left(\xbold \cdot \z_1^{\epsilon_1} \cdot \ldots \cdot \z_m^{\epsilon_m} \right)[\W(t')/\xbold] \cdot \E(\z_{m+1}[t'/\xbold])_2^{-1} \cdot \E(\z_{m+1}[t'/\xbold])_1^{\epsilon_{m+1}} \cdot \E(\z_{m+1}[t'/\xbold])_2 \\
&\equiv \E(t')_2 \cdot \left(\xbold \cdot \z_1^{\epsilon_1} \cdot \ldots \cdot \z_m^{\epsilon_m} \right)[\W(t')/\xbold] \cdot \E(t')_2^{-1} \cdot \E(t')_1^{\epsilon_{m+1}} \cdot \E(t')_2 \\ 
&= \E(t')_2 \cdot \left(\xbold \cdot \z_1^{\epsilon_1} \cdot \ldots \cdot \z_m^{\epsilon_m} \right)[\W(t')/\xbold] \cdot \E(t')_2^{-1} \cdot \xbold^{\epsilon_{m+1}} \cdot \E(t')_2 \\
&\sim \E(t')_2 \cdot \left(\xbold \cdot \z_1^{\epsilon_1} \cdot \ldots \cdot \z_m^{\epsilon_m} \right)[\W(t')/\xbold] \cdot \left(\E(t')_2^{-1} \cdot \xbold \cdot \E(t')_2\right)^{\epsilon_{m+1}} \\
&\equiv \E(t')_2 \cdot \left(\xbold \cdot \z_1^{\epsilon_1} \cdot \ldots \cdot \z_m^{\epsilon_m} \right)[\W(t')/\xbold] \cdot \W(t')^{\epsilon_{m+1}} \\ 
&\equiv \E(t')_2 \cdot \left(\xbold \cdot \z_1^{\epsilon_1} \cdot \ldots \cdot \z_{m+1}^{\epsilon_{m+1}} \right)[\W(t')/\xbold] \\
&\sim \E(t')_2 \cdot \E(t)_2[\W(t')/\xbold],     
\end{align*}
as desired. The third equality follows by assumption on $t'$, the fourth equality follows because for any group $G$ and $g, h \in G$ and $z = \pm 1$ we have $(g^{-1}hg)^z = g^{-1}h^zg$, the fifth equality follows by definition of $\W(t')$, the sixth equality follows because $\z_{m+1} = \xbold$, and the last congruence follows by Lemma \ref{canonicalformlemma}. This completes the induction and hence the proof.
\end{itemize}
\end{proof}

\noindent We finally require the following two technical lemmas about reduced words in free groups. 

\begin{lemma}
\label{firstreducedwordlemma}
Let $s \in \Term^c(\Sigma_\Grp(\xbold, \y_1, \ldots, \y_n))$ be reduced. Then in the free group on the set $\{\xbold_0, \xbold_1, \y_1, \ldots, \y_n\}$, the following claims hold:

\begin{enumerate}[(i)]

\item If $s \equiv e$, then the unique reduced word obtained from $\xbold_1 \: \cdot \: s[\xbold_1^{-1} \xbold_0 \xbold_1/\xbold]$ ends in $\xbold_1$.

\item If $s$ ends in $\xbold$, then the unique reduced word obtained from $\xbold_1 \cdot s[\xbold_1^{-1} \xbold_0 \xbold_1/\xbold]$ ends in $\xbold_0\xbold_1$.

\item If $s$ ends in $\xbold^{-1}$, then the unique reduced word obtained from $\xbold_1 \cdot s[\xbold_1^{-1} \xbold_0 \xbold_1/\xbold]$ ends in $\xbold_0^{-1}\xbold_1$.

\item If $s$ ends in $\y_i^{\epsilon}$ for some $1 \leq i \leq n$ and $\epsilon = \pm 1$, then the unique reduced word obtained from $\xbold_1 \cdot s[\xbold_1^{-1} \xbold_0 \xbold_1/\xbold]$ ends in $\xbold_1 \cdot t$ for some reduced word $t \in \Term^c(\Sigma_\Grp(\y_1, \ldots, \y_n))$ that ends in $\y_i^{\epsilon}$. 

\end{enumerate}

\end{lemma}

\begin{proof}
We prove this by induction on the length of $s$. 
\begin{itemize}
\item If $s \equiv e$, then (i) clearly holds. 

\item If $s \equiv \xbold$, then we have 
\[ \xbold_1 \cdot s[\xbold_1^{-1} \xbold_0 \xbold_1/\xbold] \equiv \xbold_1 \cdot \xbold_1^{-1} \xbold_0 \xbold_1 \sim \xbold_0 \xbold_1, \]
as desired for (ii). 

\item If $s \equiv \xbold^{-1}$, then we have 
\[ \xbold_1 \cdot s[\xbold_1^{-1} \xbold_0 \xbold_1/\xbold] \equiv \xbold_1 \cdot (\xbold_1^{-1} \xbold_0 \xbold_1)^{-1} \sim \xbold_1 \xbold_1^{-1} \xbold_0^{-1} \xbold_1 \sim \xbold_0^{-1} \xbold_1, \]
as required for (iii). 

\item If $s \equiv \y_i^{\epsilon}$ for some $1 \leq i \leq n$ and $\epsilon = \pm 1$, then we have
\[ \xbold_1 \cdot s[\xbold_1^{-1} \xbold_0 \xbold_1/\xbold] \equiv \xbold_1 \cdot \y_i^{\epsilon}, \] 
as desired for (iv).

\item Now let $s \in \Term^c(\Sigma_\Grp(\xbold, \y_1, \ldots, \y_n))$ be reduced of length $n$ for some $n \geq 1$, and assume that the result holds for $s$. We show that the result holds for $s \cdot \xbold^{\pm 1}$ and $s \cdot \y_i^{\pm 1}$ (for any $1 \leq i \leq n$), assuming that these words are reduced. 
\begin{itemize}
\item First we consider $s \cdot \xbold$. If this word is reduced, then $s$ does not end with $\xbold^{-1}$. So then $s$ ends with either $\xbold$ or $\y_i^{\epsilon}$ for some $1 \leq i \leq n$ and $\epsilon = \pm 1$. 

Suppose first that $s$ ends with $\xbold$. Then by the induction hypothesis, there is some reduced (possibly empty) word $t \in \Term^c(\Sigma_\Grp(\xbold, \y_1, \ldots, \y_n))$ such that 
\[ \xbold_1 \cdot s[\xbold_1^{-1} \xbold_0 \xbold_1/\xbold] \sim t \cdot \xbold_0 \xbold_1. \]
So then we have 
\begin{align*}
\xbold_1 \cdot (s \cdot \xbold)[\xbold_1^{-1} \xbold_0 \xbold_1/\xbold]	&\sim \xbold_1 \cdot s[\xbold_1^{-1} \xbold_0 \xbold_1/\xbold] \cdot \xbold_1^{-1} \xbold_0 \xbold_1 \\
								&\sim t \cdot \xbold_0 \xbold_1 \cdot \xbold_1^{-1} \xbold_0 \xbold_1 \\
								&\sim t \cdot \xbold_0 \xbold_0 \xbold_1,
\end{align*}
so that the reduced word obtained from $\xbold_1 \cdot (s \cdot \xbold)[(\xbold_1^{-1} \xbold_0 \xbold_1)/\xbold]$ ends in $\xbold_0 \xbold_1$, as desired for (ii). \par

Now suppose that $s$ ends with $\y_i^{\epsilon}$ for some $1 \leq i \leq n$ and $\epsilon = \pm 1$. Then by the induction hypothesis, there is some reduced (possibly empty) word $t \in \Term^c(\Sigma_\Grp(\xbold, \y_1, \ldots, \y_n))$ with
\[ \xbold_1 \cdot s[\xbold_1^{-1} \xbold_0 \xbold_1/\xbold] \sim t \cdot \xbold_1 \cdot t' \]
for some reduced $t' \in \Term^c(\Sigma_\Grp(\y_1, \ldots, \y_n))$ that ends in $\y_i^{\epsilon}$.
So then we have 
\begin{align*}
\xbold_1 \cdot (s \cdot \xbold)[\xbold_1^{-1} \xbold_0 \xbold_1/\xbold]	&\sim \xbold_1 \cdot s[\xbold_1^{-1} \xbold_0 \xbold_1/\xbold] \cdot \xbold_1^{-1} \xbold_0 \xbold_1 \\
								&\sim t \cdot \xbold_1 \cdot t' \cdot \xbold_1^{-1} \xbold_0 \xbold_1,
\end{align*}
so that the reduced word obtained from $\xbold_1 \cdot (s \cdot \xbold)[(\xbold_1^{-1} \xbold_0 \xbold_1)/\xbold]$ again ends in $\xbold_0 \xbold_1$, as desired for (ii). This completes the case for $s \cdot \xbold$. \par

\item Now we consider $s \cdot \xbold^{-1}$. If this word is reduced, then $s$ does not end with $\xbold$. So then $s$ ends with either $\xbold^{-1}$ or $\y_i^{\epsilon}$ for some $1 \leq i \leq n$ and $\epsilon = \pm 1$. 

Suppose first that $s$ ends with $\xbold^{-1}$. Then by the induction hypothesis, there is some reduced (possibly empty) word $t \in \Term^c(\Sigma_\Grp(\xbold, \y_1, \ldots, \y_n))$ with 
\[ \xbold_1 \cdot s[\xbold_1^{-1} \xbold_0 \xbold_1/\xbold] \sim t \cdot \xbold_0^{-1} \xbold_1. \]
So then we have 
\begin{align*}
\xbold_1 \cdot (s \cdot \xbold^{-1})[\xbold_1^{-1} \xbold_0 \xbold_1/\xbold]	&\sim \xbold_1 \cdot s[\xbold_1^{-1} \xbold_0 \xbold_1/\xbold] \cdot \xbold_1^{-1} \xbold_0^{-1} \xbold_1 \\
									&\sim t \cdot \xbold_0^{-1} \xbold_1 \cdot \xbold_1^{-1} \xbold_0^{-1} \xbold_1 \\
									&\sim t \cdot \xbold_0^{-1} \xbold_0^{-1} \xbold_1,
\end{align*}
so that the reduced word obtained from $\xbold_1 \cdot (s \cdot \xbold^{-1})[\xbold_1^{-1} \xbold_0 \xbold_1/\xbold]$ ends in $\xbold_0^{-1} \xbold_1$, as desired for (iii). \par

Now suppose that $s$ ends with $\y_i^{\epsilon}$ for some $1 \leq i \leq n$ and $\epsilon = \pm 1$. Then by the induction hypothesis, there is some reduced (possibly empty) word $t \in \Term^c(\Sigma_\Grp(\xbold, \y_1, \ldots, \y_n))$ with 
\[ \xbold_1 \cdot s[\xbold_1^{-1} \xbold_0 \xbold_1/\xbold] \sim t \cdot \xbold_1 \cdot t' \]
for some reduced $t' \in \Term^c(\Sigma_\Grp(\y_1, \ldots, \y_n))$ that ends with $\y_i^{\epsilon}$. Then we have 
\begin{align*}
\xbold_1 \cdot (s \cdot \xbold^{-1})[\xbold_1^{-1} \xbold_0 \xbold_1/\xbold]	&\sim \xbold_1 \cdot s[\xbold_1^{-1} \xbold_0 \xbold_1/\xbold] \cdot \xbold_1^{-1} \xbold_0^{-1} \xbold_1 \\
									&\sim t \cdot \xbold_1 \cdot t' \cdot \xbold_1^{-1} \xbold_0^{-1} \xbold_1,
\end{align*}
so that the reduced word obtained from $\xbold_1 \cdot (s \cdot \xbold^{-1})[\xbold_1^{-1} \xbold_0 \xbold_1/\xbold]$ again ends in $\xbold_0^{-1} \xbold_1$, as desired for (iii). This completes the case for $s \cdot \xbold^{-1}$. \par

\item Lastly we consider $s \cdot \y_i^{\epsilon}$ for any $1 \leq i \leq n$ and $\epsilon = \pm 1$. If this word is reduced, then $s$ does not end with $\y_i^{-\epsilon}$ (equating $-(-1)$ with $1$). So then $s$ ends with $\xbold$, with $\xbold^{-1}$, with $\y_i^{\epsilon}$, or with $\y_j^{\delta}$ for any $1 \leq j \neq i \leq n$ and $\delta = \pm 1$. \par

If $s$ ends with $\xbold$, then by the induction hypothesis, there is some reduced (possibly empty) word $t \in \Term^c(\Sigma_\Grp(\xbold, \y_1, \ldots, \y_n))$ with 
\[ \xbold_1 \cdot s[\xbold_1^{-1} \xbold_0 \xbold_1/\xbold] \sim t \cdot \xbold_0 \xbold_1. \]
So then we have 
\begin{align*}
\xbold_1 \cdot (s \cdot \y_i^{\epsilon})[\xbold_1^{-1} \xbold_0 \xbold_1/\xbold]	&\sim \xbold_1 \cdot s[\xbold_1^{-1} \xbold_0 \xbold_1/\xbold] \cdot \y_i^{\epsilon} \\
								&\sim t \cdot \xbold_0 \xbold_1 \cdot \y_i^{\epsilon},
\end{align*}
so that the reduced word obtained from $\xbold_1 \cdot (s \cdot \y_i^{\epsilon})[\xbold_1^{-1} \xbold_0 \xbold_1/\xbold]$ ends in $\xbold_1 \y_i^{\epsilon}$, as desired for (iv). Exactly similar reasoning works for the case where $s$ ends with $\xbold^{-1}$. \par

Now suppose that $s$ ends with $\y_i^{\epsilon}$. Then by the induction hypothesis, there is some reduced (possibly empty) word $t \in \Term^c(\Sigma_\Grp(\xbold, \y_1, \ldots, \y_n))$ with 
\[ \xbold_1 \cdot s[\xbold_1^{-1} \xbold_0 \xbold_1/\xbold] \sim t \cdot \xbold_1 \cdot t' \]
for some reduced $t' \in \Term^c(\Sigma_\Grp(\y_1, \ldots, \y_n))$ that ends with $\y_i^{\epsilon}$. Then we have 
\begin{align*}
\xbold_1 \cdot (s \cdot \y_i^{\epsilon})[\xbold_1^{-1} \xbold_0 \xbold_1/\xbold]	&\sim \xbold_1 \cdot s[\xbold_1^{-1} \xbold_0 \xbold_1/\xbold] \cdot \y_i^{\epsilon} \\
										&\sim t \cdot \xbold_1 \cdot t' \cdot \y_i^{\epsilon},
\end{align*}
so that the reduced word obtained from $\xbold_1 \cdot (s \cdot \y_i^{\epsilon})[\xbold_1^{-1} \xbold_0 \xbold_1/\xbold]$ has the form required for (iv). 

Finally, suppose that $s$ ends with $\y_j^{\delta}$ for some $1 \leq j \neq i \leq n$ and $\delta = \pm 1$. Then by the induction hypothesis, there is some reduced (possibly empty) word $t \in \Term^c(\Sigma_\Grp(\xbold, \y_1, \ldots, \y_n))$ with 
\[ \xbold_1 \cdot s[\xbold_1^{-1} \xbold_0 \xbold_1/\xbold] \sim t \cdot \xbold_1 \cdot t' \]
for some reduced $t' \in \Term^c(\Sigma_\Grp(\y_1, \ldots, \y_n))$ that ends with $\y_j^{\delta}$. So we have 
\begin{align*}
\xbold_1 \cdot (s \cdot \y_i^{\epsilon})[\xbold_1^{-1} \xbold_0 \xbold_1/\xbold]	&\sim \xbold_1 \cdot s[\xbold_1^{-1} \xbold_0 \xbold_1/\xbold] \cdot \y_i^{\epsilon} \\
										&\sim t \cdot \xbold_1 \cdot t' \cdot \y_i^{\epsilon},
\end{align*}
so that the reduced word obtained from $\xbold_1 \cdot (s \cdot \y_i^{\epsilon})[\xbold_1^{-1} \xbold_0 \xbold_1/\xbold]$ again has the form required for (iv), because $j \neq i$ and hence $t \cdot \xbold_1 \cdot t' \cdot \y_i^{\epsilon}$ is reduced and ends with $\xbold_1 \cdot t''$ for some $t'' \in \Term^c(\Sigma_\Grp(\y_1, \ldots, \y_n))$  that ends with $\y_i^{\epsilon}$. \par
\end{itemize}
This completes the induction and hence the proof of the lemma.
\end{itemize}
\end{proof}

\begin{lemma}
\label{secondreducedwordlemma}
Let $s \in \Term^c(\Sigma_\Grp(\xbold, \y_1, \ldots, \y_n))$ be reduced, and assume that the congruence
\[ \xbold_1 \cdot s[\xbold_1^{-1} \xbold_0 \xbold_1/\xbold] \sim s[\xbold_0/\xbold] \cdot s[\xbold_1/\xbold]^{-1} \cdot \xbold_1 \cdot s[\xbold_1/\xbold] \]
holds in the free group on the set $\{\xbold_0, \xbold_1, \y_1, \ldots, \y_n\}$. Then all occurrences of $\xbold$ in $s$ must precede all occurrences of $\y_1, \ldots, \y_n$ in $s$. 
\end{lemma}

\begin{proof}
Suppose towards a contradiction that $s$ satisfies the assumptions but contains an occurrence of $\y_i$ (for some $1 \leq i \leq n$) to the left of some occurrence of $\xbold$. Then there are $\epsilon = \pm 1$ and reduced (possibly empty) words $s_1, s_2, s_3 \in \Term^c(\Sigma_\Grp(\xbold, \y_1, \ldots, \y_n))$ with
\[ s \equiv s_1 \y_i^{\epsilon} s_2 \xbold s_3 \] or \[ s \equiv s_1 \y_i^{\epsilon} s_2 \xbold^{-1} s_3. \] Suppose first that $s \equiv s_1 \y_i^{\epsilon} s_2 \xbold s_3$. Since $s$ is reduced, it follows that $s_1$ does not end in $\y_i^{-\epsilon}$, that $s_2$ does not start with $\y_i^{-\epsilon}$ or end with $\xbold^{-1}$, and that $s_3$ does not start with $\xbold^{-1}$. The assumption on $s$ then implies that
\[ \xbold_1 \cdot s_1[\xbold_1^{-1} \xbold_0 \xbold_1/\xbold] \cdot \y_i^{\epsilon} \cdot s_2[\xbold_1^{-1} \xbold_0 \xbold_1/\xbold] \cdot \xbold_1^{-1} \xbold_0 \xbold_1 \cdot s_3[\xbold_1^{-1} \xbold_0 \xbold_1/\xbold] \] 
\[ \sim \] 
\[ s_1[\xbold_0/\xbold] \y_i^{\epsilon} s_2[\xbold_0/\xbold] \xbold_0 s_3[\xbold_0/\xbold] \cdot s_3^{-1}[\xbold_1/\xbold] \xbold_1^{-1} s_2^{-1}[\xbold_1/\xbold] \y_i^{-\epsilon} s_1^{-1}[\xbold_1/\xbold] \] \[ \cdot \ \xbold_1 \cdot s_1[\xbold_1/\xbold] \y_i^{\epsilon} s_2[\xbold_1/\xbold] \xbold_1 s_3[\xbold_1/\xbold]. \]
If $s_1$ is the empty word, then the reduction of the top word will begin with $\xbold_1 \y_i^{\epsilon}$, while the reduction of the bottom word will begin with just $\y_i^{\epsilon}$, which is impossible, since the reductions of these words are congruent in the free group on $\{\xbold_0, \xbold_1, \y_1, \ldots, \y_n\}$ and hence must be identical. \par

If $s_1$ is non-empty and ends with $\xbold$, then since $s_1$ is reduced, it will follow from Lemma \ref{firstreducedwordlemma} that the reduction of the top word will begin with $t \cdot \xbold_0 \xbold_1 \y_i^{\epsilon}$ for some reduced $t \in \Term^c(\Sigma_\Grp(\xbold_0, \xbold_1, \y_1, \ldots, \y_n))$. In particular, the reduced word obtained from the top word will have an occurrence of $\xbold_1$ before the first occurrence of $\y_i^{\epsilon}$. However, the reduced word obtained from the bottom word will \emph{not} have any occurrences of $\xbold_1$ before the first occurrence of $\y_i^{\epsilon}$, which is impossible for the reason given in the last paragraph. If $s_1$ is non-empty and ends with $\xbold^{-1}$, or with $\y_i^{\epsilon}$, or with $\y_j^{\delta}$ for some $1 \leq j \neq i \leq n$ and $\delta = \pm 1$, then exactly similar reasoning (with the use of Lemma \ref{firstreducedwordlemma}) leads to a contradiction. \par

This proves that we cannot have $s \equiv s_1 \y_i^{\epsilon} s_2 \xbold s_3$, and parallel reasoning also shows that we cannot have $s \equiv s_1 \y_i^{\epsilon} s_2 \xbold^{-1} s_3$ either, which contradicts the original assumption. So it follows that all occurrences of $\xbold$ in $s$ must precede all occurrences of $\y_1, \ldots, \y_n$ in $s$, as desired. 
\end{proof}

\noindent We can finally give a characterization of the logical isotropy groups of the free, finitely generated racks. First, the following result was proven in \cite[Section 4.1]{SDworld}:

\begin{theorem}[Dehornoy \cite{SDworld}]
\label{rackstheorem}
For any (finite) set $X$ and $s, t \in \Term^c(\Sigma(X))$, let $\E(s) = (\xbold, \omega)$ and $\E(t) = (\xbold', \omega')$ for some $\xbold, \xbold' \in X$ and $\omega, \omega' \in \Term^c(\Sigma_\Grp(X))$. Then \[ \T_\Rack(X) \vdash s = t \] iff \[ \xbold = \xbold' \ \text{and} \ \omega \sim \omega. \] \qed
\end{theorem}

\begin{theorem}[\textbf{Isotropy Group of a Free Rack}]
\label{rackisotropy}
Let $\R_n$ be the free rack on $n$ generators $\y_1, \ldots, \y_n$. For any $t \in \Term^c(\Sigma_\Grp(\xbold, \y_1, \ldots, \y_n))$, we have \[ [t] \in G_{\T_\Rack}(\R_n) \] iff there are $p, m \geq 0$ and $1 \leq i_1, \ldots, i_m \leq n$ such that
\[ [t] = \left[\xbold \lhd^{\delta_1} \ldots \lhd^{\delta_p} \xbold \lhd^{\epsilon_1} \y_{i_1} \lhd^{\epsilon_2} \ldots \lhd^{\epsilon_m} \y_{i_m}\right], \]
with $\delta_j = \pm 1$ for all $1 \leq j \leq p$ and $\epsilon_k = \pm 1$ for all $1 \leq k \leq m$, and the corresponding term $\xbold^{\delta_1} \ldots \xbold^{\delta_p} \y_{i_1}^{\epsilon_1} \ldots \y_{i_m}^{\epsilon_m} \in \Term^c(\Sigma_\Grp(\xbold, \y_1, \ldots, \y_n))$ is reduced. 
\end{theorem}

\begin{proof}
First we prove that if $t$ has the stated form, then $[t] \in G_{\T_\Rack}(\R_n)$. So let $t$ have the form described in the statement of the theorem. We must show that $[t]$ is invertible and commutes generically with the rack operations. \par

To show that $[t]$ is invertible, consider the term 
\[ t^{-1} := \xbold \lhd^{-\delta_p} \ldots \lhd^{-\delta_1} \xbold \lhd^{-\epsilon_m} \y_{i_m} \lhd^{-\epsilon_{m-1}} \ldots \lhd^{-\epsilon_1} \y_{i_1}. \]
To show that $t[t^{-1}/\xbold] \sim_\Rack \xbold$ and $t^{-1}[t/\xbold] \sim_\Rack \xbold$ in the free rack on $\{\xbold, \y_1, \ldots, \y_n\}$, it suffices by Theorem \ref{rackstheorem} to show that \[ \E(t[t^{-1}/\xbold])_1 = \E(\xbold)_1 = \E(t^{-1}[t/\xbold])_1 \] and \[ \E(t[t^{-1}/\xbold])_2 \sim \E(\xbold)_2 \sim \E(t^{-1}[t/\xbold])_2. \] We have $\E(\xbold) = (\xbold, e)$, and by Lemma \ref{canonicalformlemma} we have
\[ \E(t)_1 = \xbold = \E(t^{-1})_1 \] as well as
\[ \E(t)_2 \sim \xbold^{\delta_1} \ldots \xbold^{\delta_p} \y_{i_1}^{\epsilon_1} \ldots \y_{i_m}^{\epsilon_m} \]
and
\[ \E(t^{-1})_2 \sim \xbold^{-\delta_p} \ldots \xbold^{-\delta_1} \y_{i_m}^{-\epsilon_m} \ldots \y_{i_1}^{-\epsilon_1}. \] 
Note also that 
\begin{align*}
\W(t^{-1}) 	&:= \E(t^{-1})_2^{-1} \cdot \E(t^{-1})_1 \cdot \E(t^{-1})_2 \\
			&\sim \y_{i_1}^{\epsilon_1} \ldots \y_{i_m}^{\epsilon_m} \xbold^{\delta_1} \ldots \xbold^{\delta_p} \xbold \xbold^{-\delta_p} \ldots \xbold^{-\delta_1} \y_{i_m}^{-\epsilon_m} \ldots \y_{i_1}^{-\epsilon_1} \\
			&\sim \y_{i_1}^{\epsilon_1} \ldots \y_{i_m}^{\epsilon_m} \xbold^{\delta_1 + \ldots + \delta_p} \xbold \xbold^{-\delta_p + \ldots + -\delta_1} \y_{i_m}^{-\epsilon_m} \ldots \y_{i_1}^{-\epsilon_1} \\
			&\sim \y_{i_1}^{\epsilon_1} \ldots \y_{i_m}^{\epsilon_m} \xbold \y_{i_m}^{-\epsilon_m} \ldots \y_{i_1}^{-\epsilon_1}. 
\end{align*}
Then by Lemma \ref{mainracksubstlemma}, we have
\[ \E(t[t^{-1}/\xbold])_1 = \xbold \] and 
\begin{align*}
&\quad \ \E(t[t^{-1}/\xbold])_2 \\	
&\sim \E(t^{-1})_2 \cdot \E(t)_2[\W(t^{-1})/\xbold] \\
			&\sim \xbold^{-\delta_p} \ldots \xbold^{-\delta_1} \y_{i_m}^{-\epsilon_m} \ldots \y_{i_1}^{-\epsilon_1} \cdot \left(\xbold^{\delta_1} \ldots \xbold^{\delta_p} \y_{i_1}^{\epsilon_1} \ldots \y_{i_m}^{\epsilon_m}\right)[\W(t^{-1})/\xbold] \\
			&\sim \xbold^{-\delta_p} \ldots \xbold^{-\delta_1} \y_{i_m}^{-\epsilon_m} \ldots \y_{i_1}^{-\epsilon_1} \cdot \left(\xbold^{\delta_1 + \ldots + \delta_p} \y_{i_1}^{\epsilon_1} \ldots \y_{i_m}^{\epsilon_m}\right)[\W(t^{-1})/\xbold] \\
			&\sim \xbold^{-\delta_p} \ldots \xbold^{-\delta_1} \y_{i_m}^{-\epsilon_m} \ldots \y_{i_1}^{-\epsilon_1} \cdot \left(\xbold^{\delta_1 + \ldots + \delta_p} \y_{i_1}^{\epsilon_1} \ldots \y_{i_m}^{\epsilon_m}\right)[\y_{i_1}^{\epsilon_1} \ldots \y_{i_m}^{\epsilon_m} \xbold \y_{i_m}^{-\epsilon_m} \ldots \y_{i_1}^{-\epsilon_1}/\xbold] \\
			&\sim \xbold^{-\delta_p} \ldots \xbold^{-\delta_1} \y_{i_m}^{-\epsilon_m} \ldots \y_{i_1}^{-\epsilon_1} \cdot \left(\left(\y_{i_1}^{\epsilon_1} \ldots \y_{i_m}^{\epsilon_m} \xbold \y_{i_m}^{-\epsilon_m} \ldots \y_{i_1}^{-\epsilon_1}\right)^{\delta_1 + \ldots + \delta_p} \y_{i_1}^{\epsilon_1} \ldots \y_{i_m}^{\epsilon_m}\right) \\
			&\sim \xbold^{-\delta_p} \ldots \xbold^{-\delta_1} \y_{i_m}^{-\epsilon_m} \ldots \y_{i_1}^{-\epsilon_1} \cdot \y_{i_1}^{\epsilon_1} \ldots \y_{i_m}^{\epsilon_m} \xbold^{\delta_1 + \ldots + \delta_p} \y_{i_m}^{-\epsilon_m} \ldots \y_{i_1}^{-\epsilon_1} \y_{i_1}^{\epsilon_1} \ldots \y_{i_m}^{\epsilon_m} \\
			&\sim \xbold^{-\delta_p + \ldots + -\delta_1} \xbold^{\delta_1 + \ldots + \delta_p} \\
			&\sim e, 
\end{align*}
as desired, where the sixth congruence holds because in any group $G$ we have $\left(ghg^{-1}\right)^n = gh^ng^{-1}$ for any $g, h \in G$ and $n \in \mathbb{Z}$. The proof that $\E(t^{-1}[t/\xbold])_2 \sim e$ is similar. It follows that $[t]$ is invertible, as desired. \par

Now we show that $[t]$ commutes generically with the rack operations. So we must show that the following congruences hold in the free rack on $\{\xbold_0, \xbold_1, \y_1, \ldots, \y_n\}$:
\[ t[\xbold_0 \lhd \xbold_1/\xbold] \sim_\Rack t[\xbold_0/\xbold] \lhd t[\xbold_1/\xbold] \]
and
\[ t[\xbold_0 \lhd^{-1} \xbold_1/\xbold] \sim_\Rack t[\xbold_0/\xbold] \lhd^{-1} t[\xbold_1/\xbold]. \]
Since the proofs are similar, we only show the first. By Theorem \ref{rackstheorem}, it suffices to show that 
\[ \E(t[\xbold_0 \lhd \xbold_1/\xbold])_1 = \E(t[\xbold_0/\xbold] \lhd t[\xbold_1/\xbold])_1 \] and
\[ \E(t[\xbold_0 \lhd \xbold_1/\xbold])_2 \sim \E(t[\xbold_0/\xbold] \lhd t[\xbold_1/\xbold])_2. \]
By Lemma \ref{racksubstlemma}, since $\E(t)_1 = \xbold$, we have
\[ \E(t[\xbold_0 \lhd \xbold_1)/\xbold])_1 = \xbold_0 \] and
\begin{align*}
\E(t[\xbold_0 \lhd \xbold_1/\xbold])_2	&\sim \xbold_1 \cdot \E(t)_2[\xbold_1^{-1}\xbold_0\xbold_1/\xbold] \\
					&\sim \xbold_1 \cdot \left(\xbold^{\delta_1 + \ldots + \delta_p}\y_{i_1}^{\epsilon_1} \ldots \y_{i_m}^{\epsilon_m}\right)[\xbold_1^{-1}\xbold_0\xbold_1/\xbold] \\
					&\sim \xbold_1 \cdot \xbold_1^{-1} \xbold_0^{\delta_1 + \ldots + \delta_p} \xbold_1 \y_{i_1}^{\epsilon_1} \ldots \y_{i_m}^{\epsilon_m} \\
					&\sim \xbold_0^{\delta_1 + \ldots + \delta_p} \xbold_1 \y_{i_1}^{\epsilon_1} \ldots \y_{i_m}^{\epsilon_m},
\end{align*}
where the third congruence again holds because of the previously mentioned group-theoretic fact. Then, given that
\[ \E(t[\xbold_0/\xbold])_1 = \xbold_0 \] and
\[ \E(t[\xbold_0/\xbold])_2 \sim \xbold_0^{\delta_1 + \ldots + \delta_p} \y_{i_1}^{\epsilon_1} \ldots \y_{i_m}^{\epsilon_m} \]
and
\[ \E(t[\xbold_1/\xbold])_1 = \xbold_1 \] and
\[ \E(t[\xbold_1/\xbold])_2 \sim \xbold_1^{\delta_1 + \ldots + \delta_p} \y_{i_1}^{\epsilon_1} \ldots \y_{i_m}^{\epsilon_m}, \]
we obtain 
\[ \E(t[\xbold_0/\xbold] \lhd t[\xbold_1/\xbold])_1 = \E(t[\xbold_0/\xbold])_1 = \xbold_0 = \E(t[\xbold_0 \lhd \xbold_1/\xbold])_1 \] and
\begin{align*}
&\quad \ \E(t[\xbold_0/\xbold] \lhd t[\xbold_1/\xbold])_2 \\
&= \xbold_0^{\delta_1 + \ldots + \delta_p} \y_{i_1}^{\epsilon_1} \ldots \y_{i_m}^{\epsilon_m} \cdot \y_{i_m}^{-\epsilon_m} \ldots \y_{i_1}^{-\epsilon_1} \xbold_1^{-\delta_1 + \ldots + -\delta_p} \cdot \xbold_1 \cdot \xbold_1^{\delta_1 + \ldots + \delta_p} \y_{i_1}^{\epsilon_1} \ldots \y_{i_m}^{\epsilon_m} \\
						&\sim \xbold_0^{\delta_1 + \ldots + \delta_p} \xbold_1 \y_{i_1}^{\epsilon_1} \ldots \y_{i_m}^{\epsilon_m} \\
						&\sim \E(t[\xbold_0 \lhd \xbold_1/\xbold])_2,
\end{align*}
as required. This proves that $[t]$ commutes generically with the rack operations, which completes the proof that $[t] \in G_{\T_\Rack}(\R_n)$. \par

Now let $t \in \Term^c(\Sigma(\xbold, \y_1, \ldots, \y_n))$ with $[t] \in G_{\T_\Rack}(\R_n)$. We show that $t$ can be assumed to have the form in the statement of the theorem. \par 

First we show that $\E(t)_1 = \xbold$. Since $[t] \in G_{\T_\Rack}(\R_n)$, it follows that $[t]$ is invertible, and so there is some $s \in \Term^c(\Sigma(\xbold, \y_1, \ldots, \y_n))$ such that \[ t[s/\xbold] \sim_\Rack \xbold \sim_\Rack s[t/\xbold] \] in the free rack on $\{\xbold, \y_1, \ldots, \y_n\}$. By Theorem \ref{rackstheorem}, it then follows that $\E(t[s/\xbold])_1 = \E(\xbold)_1 = \xbold$. To show that $\E(t)_1 = \xbold$ follows from this, we first prove the following claim:

\begin{claim*}
Let $u, v \in \Term^c(\Sigma(\xbold, \y_1, \ldots, \y_n))$.

\begin{itemize}

\item If $\Left(u) = \xbold$, then $\Left(u[v/\xbold]) = \Left(v)$.

\item If $\Left(u) = \y_i$ for some $1 \leq i \leq n$, then $\Left(u[v/\xbold]) = \y_i$.
\end{itemize}
\end{claim*}

\begin{proof}
We prove this by induction on $u$ (for a fixed $v$). 
\begin{itemize}
\item If $u \equiv \xbold$, then we have $\Left(u) = \xbold$ and 
\[ \Left(u[v/\xbold]) = \Left(\xbold[v/\xbold]) = \Left(v), \] as desired.

\item If $u \equiv \y_i$ for some $1 \leq i \leq n$, then we have $\Left(u) = \y_i$ and
\[ \Left(u[v/\xbold]) = \Left(\y_i[v/\xbold]) = \Left(\y_i) = \y_i, \]
as desired. 

\item Suppose that $u \equiv u_1 \lhd u_2$ for some $u_1, u_2 \in \Term^c(\Sigma(\xbold, \y_1, \ldots, \y_n))$ for which the induction hypothesis holds. If $\Left(u_1 \lhd u_2) = \xbold$, then $\Left(u_1) = \xbold$ by definition of $\Left$. So by the induction hypothesis for $u_1$, we have $\Left(u_1[v/\xbold]) = \Left(v)$. Then we have
\[ \Left((u_1 \lhd u_2)[v/\xbold]) = \Left(u_1[v/\xbold] \lhd u_2[v/\xbold]) = \Left(u_1[v/\xbold]) = \Left(v), \]
as desired. 

If $\Left(u_1 \lhd u_2) = \y_i$ for some $1 \leq i \leq n$, then $\Left(u_1) = \y_i$ by definition of $\Left$. So by the induction hypothesis for $u_1$, we have $\Left(u_1[v/\xbold]) = \y_i$. Then we have
\[ \Left((u_1 \lhd u_2)[v/\xbold]) = \Left(u_1[v/\xbold] \lhd u_2[v/\xbold]) = \Left(u_1[v/\xbold]) = \y_i, \]
as desired.
\end{itemize}
\end{proof}

\noindent Recall from Lemma \ref{leftlemma} that $\E(u)_1 = \Left(u)$ for any $u \in \Term^c(\Sigma(\xbold, \y_1, \ldots, \y_n))$. So, given that $\E(t[s/\xbold])_1 = \xbold$, we then have $\Left(t[s/\xbold]) = \xbold$, and we want to show that $\Left(t) = \xbold$. But if we had $\Left(t) = \y_i$ for some $1 \leq i \leq n$ instead, then from the Claim it would follow that $\Left(t[s/\xbold]) = \y_i$ as well, contrary to assumption. So we must have $\Left(t) = \E(t)_1 = \xbold$, as desired. \par

Since $[t] \in G_{\T_\Rack}(\R_n)$, we know that $[t]$ commutes generically with the rack operations, and so it follows that $t[\xbold_0 \lhd \xbold_1/\xbold] \sim_\Rack t[\xbold_0/\xbold] \lhd t[\xbold_1/\xbold]$ holds in the free rack on $\{\xbold_0, \xbold_1, \y_1, \ldots, \y_n\}$, which implies by Theorem \ref{rackstheorem} that 
\[ \E(t[\xbold_0 \lhd \xbold_1/\xbold])_1 = \E(t[\xbold_0/\xbold] \lhd t[\xbold_1/\xbold])_1 \] and
\[ \E(t[\xbold_0 \lhd \xbold_1/\xbold])_2 \sim \E(t[\xbold_0/\xbold] \lhd t[\xbold_1/\xbold])_2. \] 
Then since $\E(t)_1 = \xbold$, we can use Lemma \ref{racksubstlemma} and the definition of $\E$ to reason as follows:
\begin{align*}
&\quad \ \xbold_1 \cdot \E(t)_2[\xbold_1^{-1} \xbold_0 \xbold_1/\xbold] \\	
&\sim \E(t[\xbold_0 \lhd \xbold_1/\xbold])_2 \\
&\sim \E(t[\xbold_0/\xbold] \lhd t[\xbold_1/\xbold])_2 \\
&\sim \E(t[\xbold_0/\xbold])_2 \cdot \E(t[\xbold_1/\xbold])_2^{-1} \cdot \xbold_1 \cdot \E(t[\xbold_1/\xbold])_2 \\
&\sim \E(t)_2[\xbold_0/\xbold] \cdot \E(t)_2^{-1}[\xbold_1/\xbold] \cdot \xbold_1 \cdot \E(t)_2[\xbold_1/\xbold].
\end{align*}
Now let $s := \E(t)_2$. Then the above congruence becomes 
\[ \xbold_1 \cdot s[\xbold_1^{-1} \xbold_0 \xbold_1/\xbold] \sim s[\xbold_0/\xbold] \cdot s^{-1}[\xbold_1/\xbold] \cdot \xbold_1 \cdot s[\xbold_1/\xbold], \]
Now let $s_r$ be the unique reduced word obtained from $s$, so that $s \sim s_r$ and we have 
\[ \xbold_1 \cdot s_r[\xbold_1^{-1} \xbold_0 \xbold_1 / \xbold] \sim s_r[\xbold_0/\xbold] \cdot s_r^{-1}[\xbold_1/\xbold] \cdot \xbold_1 \cdot s_r[\xbold_1/\xbold]. \]
Then by Lemma \ref{secondreducedwordlemma}, it follows that all occurrences of $\xbold$ in $s_r$ precede all occurrences of $\y_1, \ldots, \y_n$ in $s_r$. So then $s_r$ must have the form $s_r \equiv \xbold^{\delta_1} \ldots \xbold^{\delta_p} \y_{i_1}^{\epsilon_1} \ldots \y_{i_m}^{\epsilon_m}$ for some $p, m \geq 0$ and $1 \leq i_1, \ldots, i_m \leq n$, with $\delta_i = \pm 1$ for all $1 \leq i \leq p$ and $\epsilon_j = \pm 1$ for all $1 \leq j \leq m$ and $\delta_i + \delta_{i+1} \neq 0$ for all $1 \leq i < p$ (because $s_r$ is reduced). Then we have
\begin{align*}
\E(t)_2 		&= s \\
			&\sim s_r \\
			&\equiv \xbold^{\delta_1} \ldots \xbold^{\delta_p} \y_{i_1}^{\epsilon_1} \ldots \y_{i_m}^{\epsilon_m} \\ 
			&= \left(\xbold, \xbold^{\delta_1} \ldots \xbold^{\delta_p} \y_{i_1}^{\epsilon_1} \ldots \y_{i_m}^{\epsilon_m}\right)_2 \\
			&= \E\left(\xbold \lhd^{\delta_1} \ldots \lhd^{\delta_{p}} \xbold \lhd^{\epsilon_1} \y_{i_1} \lhd^{\epsilon_2} \ldots \lhd^{\epsilon_m} \y_{i_m}\right)_2,
\end{align*}	
with the last equality justified by Lemma \ref{canonicalformlemma}. Since we also have 
\[ \E(t)_1 = \xbold = \E\left(\xbold \lhd^{\delta_1} \ldots \lhd^{\delta_{p}} \xbold \lhd^{\epsilon_1} \y_{i_1} \lhd^{\epsilon_2} \ldots \lhd^{\epsilon_m} \y_{i_m}\right)_1,  \]
it follows that
\[ \E(t) = \E\left(\xbold \lhd^{\delta_1} \ldots \lhd^{\delta_{p}} \xbold \lhd^{\epsilon_1} \y_{i_1} \lhd^{\epsilon_2} \ldots \lhd^{\epsilon_m} \y_{i_m}\right)_2. \]
Then by Theorem 4, this entails that 
\[ t \sim_\Rack \xbold \lhd^{\delta_1} \ldots \lhd^{\delta_{p}} \xbold \lhd^{\epsilon_1} \y_{i_1} \lhd^{\epsilon_2} \ldots \lhd^{\epsilon_m} \y_{i_m}, \]
so that $t$ is congruent (in the free rack on $\{\xbold, \y_1, \ldots, \y_n\}$) to a term of the form described in the statement of the theorem, as desired. 	
\end{proof}

\noindent Using this characterization of the logical isotropy group of the free rack $\R_n$ on $n$ generators, we now deduce the following more \emph{algebraic} characterization:

\begin{corollary}
\label{racksisotropycor}
Let $\F_n$ and $\R_n$ be the free group and free rack on $n$ generators $\y_1, \ldots, \y_n$, respectively. Then
\[ G_{\T_\Rack}(\R_n) \cong \mathbb{Z} \times \F_n. \] 
\end{corollary}

\begin{proof}
We define a function \[ \phi : \mathbb{Z} \times \F_n \to G_{\T_\Rack}(\R_n) \] as follows. Let $\left(z, [t]\right) \in \mathbb{Z} \times \F_n$, with $t \in \Term^c(\Sigma_\Grp(\y_1, \ldots, \y_n))$. We may suppose without loss of generality that $t$ is reduced (since $t \sim t_r$, where $t_r$ is the unique reduced word congruent to $t$). So $t \equiv \y_{i_1}^{\epsilon_1} \ldots \y_{i_m}^{\epsilon_m}$ for some $m \geq 0$ and $1 \leq i_1, \ldots, i_m \leq n$ and $\epsilon_j = \pm 1$ for each $1 \leq j \leq m$ (if $m = 0$, then $t \equiv e$). If $z = 0$, then we set
\[ \phi\left(z, \left[\y_{i_1}^{\epsilon_1} \ldots \y_{i_m}^{\epsilon_m}\right]\right) := \left[\xbold \lhd^{\epsilon_1} \y_{i_1} \lhd^{\epsilon_2} \ldots \lhd^{\epsilon_m} \y_{i_m}\right] \in G_{\T_\Rack}(\R_n). \] Otherwise, we set
\[ \phi\left(z, \left[\y_{i_1}^{\epsilon_1} \ldots \y_{i_m}^{\epsilon_m}\right]\right) := \left[\xbold \lhd^{\delta_1} \ldots \lhd^{\delta_{z}} \xbold \lhd^{\epsilon_1} \y_{i_1} \lhd^{\epsilon_2} \ldots \lhd^{\epsilon_m} \y_{i_m}\right] \in G_{\T_\Rack}(\R_n), \] where $\delta_1, \ldots, \delta_z = 1$ if $z > 0$ and $\delta_1, \ldots, \delta_z = -1$ if $z < 0$. 

To see that $\phi$ is well-defined, note that if we have $s, t \in \Term^c(\Sigma_\Grp(\y_1, \ldots, \y_n))$ with $[s] = [t]$, i.e. $s \sim t$, then $s$ and $t$ will have the same (unique) reduction, and hence we will indeed have $\phi(z, [s]) = \phi(z, [t])$ for any $z \in \Z$.   

Now we show that $\phi$ is actually a group \emph{anti}-homomorphism. So let \[ \left(z, \left[\y_{i_1}^{\epsilon_1} \ldots \y_{i_m}^{\epsilon_m}\right]\right), \left(z', \left[\y_{j_1}^{\delta_1} \ldots \y_{j_p}^{\delta_p}\right]\right) \in \mathbb{Z} \times \F_n, \] with $\y_{i_1}^{\epsilon_1} \ldots \y_{i_m}^{\epsilon_m}, \y_{j_1}^{\delta_1} \ldots \y_{j_p}^{\delta_p} \in \Term^c(\Sigma_\Grp(\y_1, \ldots, \y_n))$ reduced. We want to show that 
\[ \phi\left(z + z', \left[\y_{i_1}^{\epsilon_1} \ldots \y_{i_m}^{\epsilon_m} \y_{j_1}^{\delta_1} \ldots \y_{j_p}^{\delta_p}\right]\right) = \phi\left(z', \left[\y_{j_1}^{\delta_1} \ldots \y_{j_p}^{\delta_p}\right]\right) \cdot \phi\left(z, \left[\y_{i_1}^{\epsilon_1} \ldots \y_{i_m}^{\epsilon_m}\right]\right). \]
Since the group multiplication in $G_{\T_\Rack}(\R_n)$ is given by substitution into $\xbold$, this means showing that if
\[ \phi\left(z + z', \left[\y_{i_1}^{\epsilon_1} \ldots \y_{i_m}^{\epsilon_m} \y_{j_1}^{\delta_1} \ldots \y_{j_p}^{\delta_p}\right]\right) = [t], \]
\[ \phi\left(z', \left[\y_{j_1}^{\delta_1} \ldots \y_{j_p}^{\delta_p}\right]\right) = [t_1], \] and
\[ \phi\left(z, \left[\y_{i_1}^{\epsilon_1} \ldots \y_{i_m}^{\epsilon_m}\right]\right) = [t_2], \]
then \[ t \sim_\Rack t_1[t_2/\xbold] \] holds in the free rack on $\{\xbold, \y_1, \ldots, \y_n\}$. By Theorem \ref{rackstheorem}, it suffices to show that $\E(t)_1 = \E(t_1[t_2/\xbold])_1 \in \{\xbold, \y_1, \ldots, \y_n\}$ and $\E(t)_2 \sim \E(t_1[t_2/\xbold])_2$ holds in the free group on $\{\xbold, \y_1, \ldots, \y_n\}$. By Lemma \ref{canonicalformlemma} and the definition of $\phi$, we have 
\[ \E(t)_1 = \xbold = \E(t_1)_1 = \E(t_2)_1 \] and
\[ \E(t)_2 \sim \xbold^{z+z'}\y_{i_1}^{\epsilon_1} \ldots \y_{i_m}^{\epsilon_m} \y_{j_1}^{\delta_1} \ldots \y_{j_p}^{\delta_p}, \]
\[ \E(t_1)_2 \sim \xbold^{z'} \y_{j_1}^{\delta_1} \ldots \y_{j_p}^{\delta_p}, \] 
\[ \E(t_2)_2 \sim \xbold^z \y_{i_1}^{\epsilon_1} \ldots \y_{i_m}^{\epsilon_m}. \]
Now note that 
\begin{align*}
\W(t_2) 	&:= \E(t_2)_2^{-1} \cdot \E(t_2)_1 \cdot \E(t_2)_2 \\
			&\sim \y_{i_m}^{-\epsilon_m} \ldots \y_{i_1}^{-\epsilon_1} \cdot \xbold^{-z} \cdot \xbold \cdot \xbold^{z} \y_{i_1}^{\epsilon_1} \ldots \y_{i_m}^{\epsilon_m} \\
			&\sim \y_{i_m}^{-\epsilon_m} \ldots \y_{i_1}^{-\epsilon_1} \cdot \xbold \cdot \y_{i_1}^{\epsilon_1} \ldots \y_{i_m}^{\epsilon_m}. 
\end{align*}
So then by Lemma \ref{mainracksubstlemma}, we have 
\begin{align*}
\E(t_1[t_2/\xbold])_2	&\sim \E(t_2)_2 \cdot\E(t_1)_2[\W(t_2)/\xbold] \\
&\sim \xbold^z \y_{i_1}^{\epsilon_1} \ldots \y_{i_m}^{\epsilon_m} \cdot \left(\xbold^{z'} \y_{j_1}^{\delta_1} \ldots \y_{j_p}^{\delta_p}\right)[\W(t_2)/\xbold] \\
&\sim \xbold^z \y_{i_1}^{\epsilon_1} \ldots \y_{i_m}^{\epsilon_m} \cdot \y_{i_m}^{-\epsilon_m} \ldots \y_{i_1}^{-\epsilon_1} \cdot \xbold^{z'} \cdot \y_{i_1}^{\epsilon_1} \ldots \y_{i_m}^{\epsilon_m} \cdot \y_{j_1}^{\delta_1} \ldots \y_{j_p}^{\delta_p} \\
&\sim \xbold^z \cdot \xbold^{z'} \cdot \y_{i_1}^{\epsilon_1} \ldots \y_{i_m}^{\epsilon_m} \cdot \y_{j_1}^{\delta_1} \ldots \y_{j_p}^{\delta_p} \\
&\sim \xbold^{z+z'} \cdot \y_{i_1}^{\epsilon_1} \ldots \y_{i_m}^{\epsilon_m} \cdot \y_{j_1}^{\delta_1} \ldots \y_{j_p}^{\delta_p} \\
&\sim \E(t)_2,
\end{align*}
as desired. This completes the proof that $\phi$ is an anti-homomorphism. \par

Now we show that $\phi$ is bijective. That $\phi$ is surjective follows almost immediately from Theorem \ref{rackisotropy}. To show that $\phi$ is injective, let $\left(z, \left[\y_{i_1}^{\epsilon_1} \ldots \y_{i_m}^{\epsilon_m}\right]\right) \in\mathbb{Z} \times \F_n$ with $\y_{i_1}^{\epsilon_1} \ldots \y_{i_m}^{\epsilon_m}$ reduced, and suppose that
\[ \phi\left(z, \left[\y_{i_1}^{\epsilon_1} \ldots \y_{i_m}^{\epsilon_m}\right]\right) = [\xbold], \] the unit element of the group $G_{\T_\Rack}(\R_n)$. We must show that
\[ \left(z, \left[\y_{i_1}^{\epsilon_1} \ldots \y_{i_m}^{\epsilon_m}\right]\right) = (0, [e]) \in \mathbb{Z} \times \F_n. \] By definition of $\phi$ and Theorem \ref{rackstheorem} and Lemma \ref{canonicalformlemma}, the assumption implies that
\[ \xbold^z \cdot \y_{i_1}^{\epsilon_1} \ldots \y_{i_m}^{\epsilon_m} \sim e, \] which forces $z = 0$ and $\y_{i_1}^{\epsilon_1} \ldots \y_{i_m}^{\epsilon_m} \sim e$, so that
$\left[\y_{i_1}^{\epsilon_1} \ldots \y_{i_m}^{\epsilon_m}\right] = [e]$, as desired. This proves that $\phi$ is bijective, which means that
\[ \phi : \mathbb{Z} \times \F_n \to G_{\T_\Rack}(\R_n) \]
is a group \emph{anti}-isomorphism. However, it is a simple fact of group theory that any two anti-isomorphic groups are isomorphic, and so it follows that
\[ \mathbb{Z} \times \F_n \cong G_{\T_\Rack}(\R_n), \] as desired. 
\end{proof}

We can now use our characterization(s) of the logical isotropy groups of free, finitely generated racks to deduce the following characterizations of the \emph{categorical} isotropy groups of these racks, whose proofs are similar to those of Corollary \ref{quandlesisotropycor}.

\begin{corollary}
\label{secondracksisotropycor}
Let $n \geq 0$. 
\begin{enumerate}
\item Let \[ \pi = \left(\pi_h : \cod(h) \to \cod(h)\right)_{\dom(h) = \R_n} \] be a (not necessarily natural) family of endomorphisms of racks, indexed by rack morphisms $h$ with domain $\R_n$. Then $\pi \in \Z_{\T_\Rack}(\R_n)$ iff there is a unique integer $z \in \Z$ and a unique reduced word
\[ \y_{i_1}^{\epsilon_1} \ldots \y_{i_m}^{\epsilon_m} \in \Term^c(\Sigma_\Grp(\y_1, \ldots, \y_n)) \] with the property that for any rack morphism $h : \R_n \to R$ we have
\[ \pi_h(r) = r \lhd^{\delta_1} \ldots \lhd^{\delta_z} r \lhd^{\epsilon_1} h_{i_1} \lhd^{\epsilon_2} \ldots \lhd^{\epsilon_m} h_{i_m} \in R, \] where $\delta_1, \ldots, \delta_z = 1$ if $z > 0$ and $\delta_1, \ldots, \delta_z = -1$ if $z < 0$. 

\item Let $h : \R_n \to \R_n$ be a rack endomorphism. Then $h$ is a categorical inner automorphism iff there is a unique integer $z \in \Z$ and a unique reduced word $\y_{i_1}^{\epsilon_1} \ldots \y_{i_m}^{\epsilon_m} \in \Term^c(\Sigma_\Grp(\y_1, \ldots, \y_n))$ such that 
\[ h([s]) = \left[s \lhd^{\delta_1} \ldots \lhd^{\delta_z} s \lhd^{\epsilon_1} \y_{i_1} \lhd^{\epsilon_2} \ldots \lhd^{\epsilon_m} \y_{i_m} \right] \in \R_n \] for any $[s] \in \R_n$ (so $s \in \Term^c(\Sigma(\y_1, \ldots, \y_n))$ and $\delta_1, \ldots, \delta_z$ are as above).

\item Let $h : \R_n \to \R_n$ be a rack automorphism. Then $h$ is a \textbf{categorical} inner automorphism iff $h$ is an \textbf{algebraic} inner automorphism.    
\end{enumerate} \qed
\end{corollary}

\noindent As for quandles, we can deduce a characterization of the \emph{global isotropy group} of the category $\Rack$ of racks and their homomorphisms, i.e. the group $\Aut\left(\Id_\Rack\right)$ of automorphisms of the identity functor $\Id_\Rack : \Rack \to \Rack$ (which is also the group of invertible elements of the \emph{centre} of the category $\Rack$, which is the monoid $\End\left(\Id_\Rack\right)$ of natural \emph{endo}morphisms of the identity functor). Since the category $\Rack$ has an initial object, namely the absolutely free rack $\R_0$ (whose carrier is just the empty set), it is easy to see that the global isotropy group of $\Rack$ is exactly the (covariant) categorical isotropy group of the initial object $\R_0$, i.e.
\[ \Aut\left(\Id_\Rack\right) = \mathcal{Z}_{\T_\Rack}(\R_0). \] Since
\[ \mathcal{Z}_{\T_\Rack}(\R_0) \cong G_{\T_\Rack}(\R_0) \cong \mathbb{Z} \times \F_0 \cong \mathbb{Z} \] by Corollary \ref{racksisotropycor} (since $\F_0$ is the trivial group), we thus obtain:

\begin{corollary}
\label{globalrackcor}
The global isotropy group of the category $\Rack$ is isomorphic to the group $\mathbb{Z}$:
\[ \Aut\left(\Id_\Rack\right) \cong \mathbb{Z}. \] Explicitly, the natural automorphisms of $\Id_\Rack$ are exactly the natural transformations $\psi : \Id_\Rack \to \Id_\Rack$ with
\[ \psi_R(r) = r \lhd^{\delta_1} \ldots \lhd^{\delta_z} r \tag{$r \in R$} \] for any $z \in \mathbb{Z}$ and rack $R$ (with $\delta_1, \ldots, \delta_z = \pm 1$ as in Corollary \ref{secondracksisotropycor}).
\qed 
\end{corollary}

\noindent We also note in connection with Corollary \ref{globalrackcor} that M. Szymik independently proved in \cite[Theorem 5.4]{Szymik} that the center $\End\left(\Id_\Rack\right)$ of the category $\Rack$ is also isomorphic to $\Z$, the free group on one generator. So we obtain as a further corollary:

\begin{corollary}
The global isotropy group of the category $\Rack$ is equal to its center, and both are isomorphic to $\Z$. \qed
\end{corollary}

\section{Conclusions}

We have characterized the isotropy groups of the free, finitely generated racks and quandles both logically and categorically, and shown as a consequence that the notions of \emph{categorical} and \emph{algebraic} inner automorphism coincide for such racks and quandles. One would hope to be able to extend the results herein to \emph{arbitrary} (or at least \emph{finitely presented}) racks and quandles, but it is not clear how one could accomplish this, since the results herein relied heavily on the solutions of the word problems for free racks and quandles given in terms of the solution of the word problem for free groups (\cite[Section 4.1]{SDworld}), whereas it is known that the general word problems for finitely presented racks and quandles are undecidable (\cite{Belk}). 

However, as shown in \cite{MFPSpaper} and \cite{thesis}, in order to compute the isotropy group of a quandle $\Q$, one ideally only needs an effective description of the quandle $\Q\la \xbold \ra$ obtained from $\Q$ by freely adjoining a new element $\xbold$, which is just the free product (or coproduct) of the quandle $\Q$ with the free quandle $\la \xbold \ra$ on one generator $\xbold$. In the recent work \cite{Bardakov}, it is shown for quandles $\Q_1$ and $\Q_2$ that if the canonical maps of $\Q_1$ and $\Q_2$ into their associated groups are injective, then the free product (i.e. coproduct) of $\Q_1$ and $\Q_2$ has an explicit presentation. However, despite the significant improvement of this presentation over the canonical presentation of the coproduct of models of a general equational theory, this presentation is still much more difficult to work with than the presentations of free quandles, even in the specific case where $\Q_2$ is the free quandle on one generator. We have thus far not been able to extend the results herein to quandles $\Q$ for which the canonical map into its associated group is injective.

\section*{Acknowledgements}
The research in this article was completed as part of the author's PhD thesis at the University of Ottawa. The author is grateful to his advisors Pieter Hofstra and Philip Scott for discussions about the material herein.


\medskip

\begin{thebibliography}{1}

\bibitem{Bardakov}
V. Bardakov, T. Nasybullov. 
Embeddings of quandles into groups.
Journal of Algebra and Its Applications, Vol. 19, No. 07, 2020.

\bibitem{Belk}
J. Belk, R.W. McGrail. 
The word problem for finitely presented quandles is undecidable.
Lecture Notes in Computer Science 9160, 1-13, 2015.

\bibitem{Bergman}
G. Bergman. 
An inner automorphism is only an inner automorphism, but an inner endomorphism can be
something strange. 
Publicacions Matematiques 56, 91-126, 2012.

\bibitem{SDworld}
P. Dehornoy. 
Some aspects of the SD-world.
Preprint, 2017. Available at https://arxiv.org/abs/1711.09792.

\bibitem{MFPSpaper}
P. Hofstra, J. Parker, P. J. Scott.
Isotropy of algebraic theories. 
Electronic Notes in Theoretical Computer Science 341, 201-217, 2018. 

\bibitem{Joyce}
D. Joyce.
A classifying invariant of knots, the knot quandle.
Journal of Pure and Applied Algebra 23, 37-65, 1982.

\bibitem{thesis}
J. Parker. 
Isotropy groups of quasi-equational theories.
PhD Thesis, University of Ottawa, 2020.

\bibitem{Szymik}
M. Szymik.
Permutations, power operations, and the center of the category of racks.
Communications in Algebra 46, 230-240, 2018.

\end{thebibliography}
\end{document}